\documentclass[english]{article}
\usepackage{preamble}

\begin{document}

\title{Higher semiadditive Grothendieck-Witt theory and the $K(1)$-local sphere}
\maketitle

\begin{abstract}
    We develop a higher semiadditive version of Grothendieck-Witt theory. We then apply the theory in the case of a finite field to study the higher semiadditive structure of the $K(1)$-local sphere $\Sph_{K(1)}$ at the prime $2$, in particular realizing the non-$2$-adic rational element $1+\varepsilon \in \pi_0\Sph_{K(1)}$ as a ``semiadditive cardinality.''  As a further application, we compute and clarify certain power operations in $\pi_0\Sph_{K(1)}$.    
\end{abstract}
\setcounter{tocdepth}{2}
\tableofcontents{}

\section{Introduction}

Let $M$ be an abelian group with an action of a finite group $G$.  Then one has a natural \emph{norm map} 
\[
\Nm_G: M_G \to M^G
\]
given by the formula $\overline{m} \mapsto \Sigma_g gm$.  This map is not an isomorphism in general, but it is in the special case where $M$ is a $\mathbb{Q}$-vector space.  This key feature of the rational situation makes representation theory in characteristic zero more tractable than representation theory in positive characteristic.  

In chromatic homotopy theory, one studies a sequence of $\infty$-categories which interpolate between the two situations: these are known (for a fixed prime $p$) as \emph{$K(n)$-local spectra} and denoted by $\Sp_{K(n)}$ for $n\geq 0$.   The $\infty$-category $\Sp_{K(0)}\simeq \Sp_{\mathbb{Q}}$ is simply the derived $\infty$-category of $\mathbb{Q}$, and each $\Sp_{K(n)}$ can be thought of as the localization of $p$-local spectra at a homotopical analogue of a field.  These $\infty$-categories of ``intermediate characteristic'' \cite{LurInt} share a striking similarity with the rational situation:

\begin{thm}\label{tate_vanish}[Hovey-Sadofsky \cite{HS96}, Greenlees-Sadofsky \cite{GS96}]
Let $G$ be a finite group and $X\in \Sp_{K(n)}^{BG}$ be a $K(n)$-local spectrum with $G$-action.  Then the natural norm map
\[
L_{K(n)}X_{hG} \to X^{hG}
\]
is an equivalence.  
\end{thm}

\Cref{tate_vanish} was subsequently generalized to the telescopic localizations $\Sp_{T(n)}$ by Kuhn \cite{kuhn2004tate}.  

Hopkins and Lurie \cite{HL} gave the following insightful interpretation of these ideas: recall that if $\cC$ is a semiadditive category, then for any finite set $T$ and $T$-tuple of elements $\{ Y_t \}_{t\in T}$ of $\cC$, there is a canonical isomorphism $\coprod_{t\in T} Y_t \simeq \prod_{t\in T}Y_t$.  \Cref{tate_vanish} asserts that an analogous statement is true in $\Sp_{K(n)}$ with the groupoid $BG$ in the place of the finite set $T$; namely, by viewing $X$ as a functor $X: BG \to \Sp_{K(n)}$, the theorem asserts that the colimit and the limit of $X$ are naturally homotopy equivalent.

Thus, $\Sp_{K(n)}$ exhibits a ``higher'' form of semiadditivity, which Hopkins and Lurie \cite{HL} dub \emph{$1$-semiadditivity}: the idea is that a $0$-semiadditive $\infty$-category is simply a semiadditive $\infty$-category, and for $m > 0$, an $m$-semiadditive $\infty$-category $\cD$ should have the feature that for any $m$-truncated $\pi$-finite space $A$ (that is, a space with finite homotopy groups concentrated in degrees $0$ through $m$) and any diagram $f:A \to \cD$, there is a canonical isomorphism 
\[
\Nm_f: \colim f \iso \invlim f.
\]
Hopkins and Lurie extend \Cref{tate_vanish} by showing that $\Sp_{K(n)}$ is $m$-semiadditive for all $m$.  This was further generalized in work by the first author, Schlank, and Yanovski \cite{TeleAmbi}, who proved the analogous statement for the telescopic localizations $\Sp_{T(n)}$.  

An important consequence of higher semiadditivity is that the objects in a higher semiadditive $\infty$-category acquire a rich algebraic structure.  Namely, recall that any object $Y$ in a semiadditive category canonically acquires the structure of a commutative monoid; in particular, for any finite set $T$, there is a canonical addition map 
\[
\prod_T Y \simeq \coprod_T Y \to Y,
\]
which exists because the semiadditive structure supplies an \emph{isomorphism} $\coprod_T Y \simeq \prod_T Y$.  An object $X$ of an $m$-semiadditive $\infty$-category admits even more structure: one can sum not only over finite sets, but over any $m$-truncated $\pi$-finite space $A$.  More precisely, given such an $A$, there is a canonical higher addition map $\push{A} \colon X^A \to X$ given by the composite
\[
X^A \simeq X\otimes A \to X \otimes \pt  \simeq X.
\]
Here, the first map is (the inverse of) the equivalence resulting from the higher semiadditive structure, and the second map is induced by the map $A \to \mathrm{pt}$.  

\begin{rem}\label{intro:highercomm}
The collection of these $\push{A}$ maps, together with corresponding restriction maps, endow $X$ with a structure known as an $m$-commutative monoid (with $0$-commutative monoids being just ordinary commutative monoids).  First described systematically by Harpaz \cite{HarpazSpans}, this $m$-commutative monoid structure for $X$ is captured by a functor 
\[
X^{(-)}: \Span(\Spc_{m})^{\op} \to \Sp
\]
given by $A\mapsto X^A$.  Here, $\Span(\Spc_m)$ is the $\infty$-category of $m$-truncated $\pi$-finite spaces and spans.  In fact, we will work in a $p$-typical setting, where one further restricts to $m$-truncated $\pi$-finite spaces with homotopy groups of $p$-power order.  
We refer the reader to \Cref{sect:prelim} for a more careful discussion of ($p$-typical) $m$-commutative monoids.    
\end{rem}

\subsection{Higher semiadditive cardinality}

This paper focuses on one part of this $m$-commutative monoid structure:

\begin{defn}[\cite{AmbiHeight} Definition 2.1.5, \cite{HL} Notation 5.1.7]
Let $\cC$ be an $m$-semiadditive $\infty$-category, let $X\in \cC$ and let $A$ be an $m$-truncated $\pi$-finite space.  Then let $|A|_X: X\to X$ be the endomorphism of $X$ defined by the composite 
\[
|A|_X: X \oto{\Delta_A} X^A \oto{\push{A}} X,
\]
where $\Delta_A$ is the canonical diagonal map induced by $A\to \pt$.  As $X$ varies, these maps assemble to a natural transformation $|A|: \Id_{\cC} \to \Id_{\cC}$, which we will refer to as the \emph{cardinality of $A$}.  

\end{defn}

Informally, $|A|$ is the $A$-fold sum of the identity map.  The term ``cardinality'' arises because in the special case where $A$ is discrete, the endomorphism $|A|$ is given by multiplication by the cardinality of the finite set $A$.  In fact, the cardinality of $\pi$-finite spaces in rational spectra is a well-known invariant:

\begin{example}\label{exm:rat_card}
In the $\infty$-category $\Sp_{\mathbb{Q}}$ of rational spectra, the cardinality of any $\pi$-finite space $A$ is given by the Baez-Dolan homotopy cardinality, which is the rational number
\[
|A|_0 := \sum_{a\in \pi_0(A)} \prod_{n\geq 1}|\pi_n(A,a)|^{(-1)^n},
\]
regarded as an element of $\pi_0(\Sph_{\mathbb{Q}})\simeq \mathbb{Q}$ (cf. \cite[Example 2.2.2]{AmbiHeight}).  
\end{example}

But the cardinalities of $\pi$-finite spaces determine canonical invariants in \emph{any} higher semiadditive $\infty$-category.  In particular, since $\Sp_{K(n)}$ is $\infty$-semiadditive \cite{HL}, one has a family of natural elements $|A|_{\Sph_{K(n)}} \in \pi_0(\Sph_{K(n)})$, one for each $\pi$-finite space $A$.  While these invariants are not known for general $n$, the main result of this paper is a (partial) extension of \Cref{exm:rat_card} to the case of $\Sp_{K(1)}$.  More specifically, for any prime $p$ and any $\pi$-finite space $A$ with $p$-power torsion homotopy groups, we compute the cardinality of $A$ as an element $|A|\in \pi_0 \Sph_{K(1)}$ (\Cref{thm:main}).

\subsection{Motivation: $T(n)$-local homotopy theory}

The authors' motivation for considering higher semiadditive cardinalities comes from studying the $\infty$-categories of telescopically localized spectra.  Obtained by localizing spectra at a height $n$ telescope $T(n)$, these $\infty$-categories $\Sp_{T(n)}$ admit localization functors $L_{K(n)}: \Sp_{T(n)} \to \Sp_{K(n)}$ which Ravenel's \emph{telescope conjecture} predicts are equivalences; while this conjecture is true at height 1 by work of Miller \cite{Miller}, Mahowald \cite{Mahowald}, and Bousfield \cite{Bousfield}, it is widely believed to be false for $n\geq 2$.  As a result, telescopically localized homotopy theory is notoriously difficult to access computationally.

However, by joint work between the second author, Schlank, and Yanovski \cite{TeleAmbi}, following work of Kuhn \cite{kuhn2004tate} (in the case of $1$-semiadditivity), the $\infty$-category of $T(n)$-local spectra is $\infty$-semiadditive.  It follows that one has canonical elements $|A|_{\Sph_{T(n)}} \in \pi_0(\Sph_{T(n)})$ for all $\pi$-finite spaces $A$.  Moreover, the localization functor $L_{K(n)}: \Sp_{T(n)} \to \Sp_{K(n)}$ is colimit preserving, and therefore sends the element $|A|_{\Sph_{T(n)}}$ to the element $|A|_{\Sph_{K(n)}} \in \pi_0 \Sph_{K(n)}$.  This provides a family of elements $|A|$ in the homotopy of the $K(n)$-local spheres which provably lift to elements in the $T(n)$-local spheres -- one may then wonder which elements in $\pi_0(\Sph_{K(n)})$ arise in this manner.  

\begin{example}[Example 2.2.4 \cite{AmbiHeight}] \label{exm:En-mod-cardinality}
In the $\infty$-category $\Mod_{E_n}(\Sp_{K(n)})$ of $K(n)$-local modules over a height $n$ Lubin-Tate theory $E_n$, one has the explicit formula 
\[
|B^{k}C_p| = p^{\binom{n-1}{k}}.
\]
\end{example}

In other words, under the natural map $\Sph_{T(n)} \to E_n$, the homotopy element $|B^k C_p|_{\Sph_{T(n)}}$ is taken to the integer $ p^{\binom{n-1}{k}}.$  Consequently, one may naturally ask whether the cardinalities $|A|$ produce interesting elements in the homotopy groups of $\Sph_{T(n)}$, or if they are always just rational numbers such as the ones given in \Cref{exm:En-mod-cardinality}.

A consequence of our main result is that the cardinalities of $\pi$-finite spaces \emph{can} in general produce non-rational elements in the $K(n)$ or $T(n)$-local spheres.  In particular, we show at the prime $2$ that $|BC_2|_{\Sph_{K(1)}}\in \pi_0(\Sph_{K(1)})$ is \emph{not} a $2$-adic integer.

\subsection{Outline of results}

\subsubsection{$K(1)$-local cardinalities}
\label{subsubsec:sk1}
In order to state our results, we recall some facts about the $K(1)$-local sphere, many of which date back to the work of Adams on the image of $J$; the reader is referred to \cite[\S 1, \S 2]{HopkinsK1} and \cite[\S 6]{Henn} for more details. 

Let $p$ be a prime and let $\Sph_{K(1)}$ denote the $K(1)$-local sphere at the prime $p$.  Then $\Sph_{K(1)}$ fits into a cofiber sequence 
\begin{equation}\label{eqn:cofibSK1}
\Sph_{K(1)} \to K \oto{\psi^{g}-1} K
\end{equation}
where $K$ denotes $KU^{\wedge}_p$ for $p$ odd and $KO^{\wedge}_2$ for $p=2$, and $\psi^{g}$ denotes an Adams operation corresponding to a topological generator $g\in \ZZ_p^{\times}$ if $p$ odd and $g\in \ZZ_2^{\times}/\{\pm 1\}$ if $p=2$.  From this fiber sequence we can extract the ring structure of $\pi_0\Sph_{K(1)}$. Note first that if $\eta \in \pi_1 KO$ is the Hopf element, then $\psi^g(\eta) = \eta$.  Then, by the long exact sequence of homotopy groups associated with (\ref{eqn:cofibSK1}), we get
\[
\pi_0\Sph_{K(1)} \cong 
\begin{cases}
\ZZ_p    & \text{if }p\text{ is odd},\\
\ZZ_2 \oplus \ZZ/2 \cdot \varepsilon    & \text{if }p=2.
\end{cases}
\]
Here, $\delta \colon \pi_1\KO^{\wedge}_2 \to \pi_0\Sph_{K(1)}$ is the boundary map and
\[
\varepsilon = \delta(\eta) = \eta \cdot \zeta
\] 
where $\zeta = \delta(1) \in \pi_{-1}\Sph_{K(1)}$. 
Since, again from the long exact sequence, $\pi_{-2}\Sph_{K(1)} = 0$, we deduce that $\zeta^2 = 0$ and hence $\varepsilon^2 = 0$. Therefore, we obtain the following ring structure for $\pi_0\Sph_{K(1)}$:
\begin{equation} \label{pi_0sk1}
\pi_0\Sph_{K(1)} \cong 
\begin{cases}
\ZZ_p    & \text{if }p\text{ is odd},\\
\ZZ_2[\varepsilon]/(\varepsilon^2,2\varepsilon)    & \text{if }p=2.
\end{cases}
\end{equation}

Our main result, which is proven as \Cref{card_BC_2}, is:

\begin{thmx}\label{thm:main}
Let $p$ be a prime, let $\Sph_{K(1)}$ denote the $K(1)$-local sphere at the prime $p$, and let $A$ be a connected $\pi$-finite space with $p$-power torsion homotopy groups.  Then
\[
|A|_{\Sph_{K(1)}} = 
\begin{cases}
1 & \text{if }p\text{ is odd},\\
1+\log_2(|A|_0)\cdot \varepsilon & \text{if }p=2,
\end{cases}
\]
where $|A|_0$ denotes the homotopy cardinality of \Cref{exm:rat_card}.  In particular,
\[
|BC_p|_{\Sph_{K(1)}} = 
\begin{cases}
1 & \text{if }p\text{ is odd},\\
1+\varepsilon & \text{if }p=2.
\end{cases}
\]
\end{thmx}

\Cref{thm:main} also determines $|A|_X: X \to X$ for general $X\in \Sp_{K(1)}$ as multiplication by $|A|_{\Sph_{K(1)}}$.  
In fact, the most essential content of the theorem is the computation of $|BC_2|$ in the $\Sph_{K(1)}$ at the prime $2$.  This is because of the following two remarks:

\begin{rem}
For all primes $p$, the computation of the cardinalities for spaces as in \Cref{thm:main} reduces to the computation of $|BC_p|_{\Sph_{K(1)}}$ via formal properties of higher semiadditivity (cf. \Cref{card_BC_2} for the reduction).   
\end{rem}

\begin{rem}
The map of rings $\Sph_{K(1)} \to KU^{\wedge}_p$ induces a ring map
\begin{equation}\label{eqn:pi0unit}
\pi_0\Sph_{K(1)} \to \pi_0 KU^{\wedge}_p \cong \ZZ_p.
\end{equation}
This map arises as the endomorphisms of the unit of the colimit preserving functor 
\[
-\otimes_{\Sph_{K(1)}} KU^{\wedge}_p \colon \Sp_{K(1)} \to \Mod_{KU^{\wedge}_p}(\Sp_{K(1)}),
\]
and thus, for any $\pi$-finite space $A$, it sends $|A|\in \pi_0 \Sph_{K(1)}$ to $|A|\in \pi_0 KU^{\wedge}_p.$


In the case when the prime $p$ is odd, the map (\ref{eqn:pi0unit}) is an isomorphism.  It follows that since $|BC_p|_{KU^{\wedge}_p} = 1$ (in $\Mod_{KU^{\wedge}_p}(\Sp_{K(1)})$) by \Cref{exm:En-mod-cardinality}, we also have $|BC_p|_{\Sph_{K(1)}}=1$.  

On the other hand, when $p=2$, the map of (\ref{eqn:pi0unit}) is the map 
\[
\ZZ_2[\varepsilon]/(\varepsilon^2,2\varepsilon)  \to \ZZ_2,
\]
given by the formula $a+ b\varepsilon \mapsto a.$
Since this map is not injective, the fact that $|BC_2|=1$ in $\Mod_{KU^{\wedge}_2}(\Sp_{K(1)})$ does not imply the same statement in $\Sp_{K(1)}$.  In particular, both $1$ and $1+\varepsilon$ in $\pi_0\Sph_{K(1)}$ map to $1\in \pi_0 KU^{\wedge}_2$.  The majority of the paper is dedicated to showing that in fact $|BC_2|_{\Sph_{K(1)}} = 1+ \varepsilon \in \pi_0\Sph_{K(1)}$.  
\end{rem}

\subsubsection{Higher semiadditive Grothendieck-Witt theory}

The proof of \Cref{thm:main} is inspired by the following: at an odd prime $p$, it follows from work of Quillen that one may identify $\Sph_{K(1)}$ as the $K(1)$-local algebraic $K$-theory of a finite field $\FF_{\ell}$, where $\ell$ is prime such that $\ell\in \ZZ_p^{\times}$ is a topological generator.  At the prime $p=2$, one may consider a more sophisticated variant of this construction: instead of finite dimensional $\FF_{\ell}$-vector spaces, one may consider the groupoid $\QF(\FF_{\ell})$ of pairs $(V,b)$ where $V$ is a finite dimensional $\FF_{\ell}$-vector space and $b$ is a non-degenerate symmetric bilinear form on $V$.  The group completion spectrum of $\QF(\FF_{\ell})$ is known as the Grothendieck-Witt theory of $\FF_{\ell}$, denoted $\GW(\FF_{\ell})$; this construction has been extensively generalized and studied in the foundational works \cite{nine-author,nine-author2, nine-author3}.  It follows essentially from work of Friedlander that one has the following:

\begin{thm}[Friedlander \cite{Friedlander}] \label{thm:introFriedlander}
Let $\ell \equiv 3,5\pmod{8}$ and let $\Sph_{K(1)}$ denote the $K(1)$-local sphere at the prime $p=2$.  Then there is an equivalence of spectra
\[
L_{K(1)}\GW(\FF_{\ell}) \simeq \Sph_{K(1)}.
\]
\end{thm}

Our strategy to prove \Cref{thm:main} is to model the higher semiadditive structure on $\Sph_{K(1)}$ as the effect on $K$-theory of certain explicit operations on symmetric bilinear forms.  This will allow us to compute the operation $|BC_2|_{\Sph_{K(1)}}:\Sph_{K(1)} \to \Sph_{K(1)}$ in terms of a certain concrete map
\[
\QF(\FF_{\ell}) \to \QF^{BC_2}(\FF_{\ell}) \xrightarrow{\push{BC_2}} \QF(\FF_{\ell}).
\]
Here, $\QF^{BC_2}$ denotes the $C_2$-equivariant analog of $\QF$, and the first map regards a symmetric bilinear form as a $C_2$-equivariant symmetric bilinear form with respect to the trivial action.  

More precisely, for every commutative algebra $\cC$ in $p$-typically $m$-semiadditive $\infty$-categories (in the sense of \cite[Definition 3.1.1]{AmbiHeight}), we define a spectrum $\GW(\cC)$, the Grothendieck-Witt spectrum of $\cC$.  Moreover, 
we show that this construction is compatible with higher semiadditivity in the following sense (cf. \Cref{sect:prelim} for notations and \Cref{defn:GW} for a more precise statement):

\begin{thmx}[Higher semiadditive Grothendieck-Witt theory]\label{thm:GWintro}
Let $\cC$ be a commutative algebra in $p$-typical $m$-semiadditive $\infty$-categories.  Then $\GW(\cC)$ extends canonically to a functor
\[
\GW^{(-)}(\cC):\Span(\Sfin{m}{p})^{\op} \to \Sp
\]
such that $\GW^{(\pt)}(\cC) \simeq \GW(\cC)$.  Here, $\Sfin{m}{p}$ denotes $m$-truncated spaces whose homotopy groups are finite of $p$-power order.  
\end{thmx}

In the case when $\cC$ is the category of finite dimensional $\FF_{\ell}$-vector spaces and $m=1$, the value of the functor $\GW^{(-)}(\FF_{\ell}) := \GW^{(-)}(\cC)$ at a groupoid $BG$ is the Grothendieck-Witt spectrum of $G$-equivariant symmetric bilinear forms over $\FF_{\ell}$.  One can think of this theorem as endowing these various equivariant theories of symmetric bilinear forms with coherent restriction and transfer maps.  
On the other hand, since $\Sp_{K(1)}$ is $\infty$-semiadditive, any $K(1)$-local spectrum $X$ extends to an $m$-commutative monoid
\[
X^{(-)}: \Span(\Sfin{m}{p})^\op \to \Sp_{K(1)}
\]
given by $A \mapsto X^A.$  The following theorem asserts that, for $X = \Sph_{K(1)}$ (at the prime $p=2$), this $m$-commutative monoid structure can be understood in terms of the higher semiadditive Grothendieck-Witt theory of \Cref{thm:GWintro}.  

\begin{thmx}\label{thm:sk1GWintro}
Let $\ell\equiv 3,5\pmod{8}$ be a prime, let $\GW^{(-)}(\FF_{\ell})$ be as in \Cref{thm:GWintro} (with $p=2$), and let $\Sph_{K(1)}$ denote the $K(1)$-local sphere at the prime $2$.  
Then there is a canonical equivalence
    \[
    L_{K(1)} \GW^{(-)}(\FF_{\ell}) \simeq \Sph_{K(1)}^{(-)}
    \]
    of functors  $\Span(\Sfin{1}{2})^\op \to \Sp_{K(1)}$.  
\end{thmx}

In other words, Friedlander's equivalence $L_{K(1)} \GW(\FF_{\ell}) \simeq \Sph_{K(1)}$ extends to an equivalence of ($2$-typical) $1$-commutative monoids.  We prove \Cref{thm:sk1GWintro} in \Cref{sect:gwfield}, and use it to compute $|BC_2|_{\Sph_{K(1)}}$ in \Cref{card_BC_2}, thus finishing the proof of \Cref{thm:main}.

\subsubsection{Computational consequences in $\Sph_{K(1)}$}
As an application of \Cref{thm:main}, we finish the paper by computing certain well-known operations in $\pi_0\Sph_{K(1)}$: namely, the power operations $\theta$ (first defined by McClure \cite[\S IX]{Hinfty}, \cite{HopkinsK1}) and $\delta_p$ (appearing in the work of Schlank, Yanovski, and the first author \cite{TeleAmbi}), and the $K(1)$-local logarithm $\log_{K(1)}$ of Rezk \cite{RezkLog}.  

\begin{rem}
These operations are completely computed in the literature in the case of an odd prime $p$ (this is elementary for the power operations $\delta_p$ and $\theta_p$, and done by Rezk in the case of $\log_{K(1)}$).  In the case $p=2$, there are additional ambiguities related to the torsion in $\pi_0(\Sph_{K(1)})$, and the authors were not able to find a reference for complete computations of these operations, though they may be known to experts.  For the sake of completeness and convenience to the reader, we state the following results for all primes $p$.
\end{rem}

\begin{thm}\label{thm:intropowop}
Let $p$ be a prime and let $\Sph_{K(1)}$ denote the $K(1)$-local sphere at $p$.  Then for $p$ odd and $x\in \pi_0(\Sph_{K(1)}) \cong \mathbb{Z}_p$, we have 
\[
\delta_p(x) = \theta(x) = \frac{x-x^p}{p}.
\]
For $p=2$ and an element $r+d\varepsilon \in \pi_0(\Sph_{K(1)}) \cong \mathbb{Z}_2[\varepsilon]/(2\varepsilon, \varepsilon^2)$,
\[
\delta_2(r+d\varepsilon) = \theta(r+d\varepsilon) - d\varepsilon = \frac{r-r^2}{2} + rd\varepsilon.
\]
\end{thm}


\begin{thm}\label{thm:introlog}
Let $\log_p \colon 1+ p \ZZ_p \to \ZZ_p$ denote the $p$-adic logarithm, and consider the $K(1)$-local Rezk logarithm map $\log_{K(1)} : \gl_1\Sph_{K(1)} \to \Sph_{K(1)}$ on $\pi_0$.  
\begin{itemize}
    \item For an odd prime $p$, the Rezk logarithm is given on $\pi_0 (\gl_1\Sph_{K(1)}) \cong (\pi_0\Sph_{K(1)})^\times\cong \ZZ_p^\times$ by the formula 
\[
\log_{K(1)}(x) = \frac{1}{p}\log_p(x^{p-1}).
\]
    \item For $p=2$, the Rezk logarithm is given on $\pi_0 (\gl_1\Sph_{K(1)}) \cong(\pi_0\Sph_{K(1)})^{\times}
    \cong (\ZZ_2[\varepsilon]/(2\varepsilon, \varepsilon^2))^{\times} \cong \ZZ_2^\times\oplus \ZZ/2\ZZ \varepsilon$ by the formula
    \[
    \log_{K(1)}(r + d\varepsilon) = \frac{1}{2}\log_2(r) + \frac{r-1}{2} \varepsilon.
    \]
\end{itemize}
\end{thm}

This follows from work of Rezk \cite[Theorem 1.9]{RezkLog} when $p$ is odd and up to torsion when $p=2$; additionally, work of Clausen \cite{DustinLog} computes the $p=2$ case on the $\ZZ_2^{\times}$ component in the source.  We extend the result to all of $(\pi_0\Sph_{K(1)})^{\times}$ using \Cref{thm:main} and an observation of T. Schlank that $|BC_2|_{\Sph_{K(1)}}\in \left(\pi_0\Sph_{K(1)}\right)^{\times}$ is always a \emph{strict unit}, that is, the image of $1$ under a spectrum map $\ZZ \to \mathrm{gl}_1 \Sph_{K(1)}$.

\subsubsection*{Outline of paper}
In \S \ref{sect:prelim}, we review some basic notions about higher semiadditivity.  In \S \ref{sec:higher_QG}, we construct a higher semiadditive refinement of Grothendieck-Witt theory, proving \Cref{thm:GWintro}; this section is technical and the details of the proof are not used elsewhere in the paper.  In \S \ref{sect:gwfield}, we make the core argument of the paper and prove \Cref{thm:sk1GWintro}.  Finally, in \S \ref{sect:computations}, we complete the proof of \Cref{thm:main} and deduce computations of operations on the $K(1)$-local sphere. 

\subsubsection*{Acknowledgements}
The authors would like to thank Clark Barwick, Jacob Lurie, Tomer Schlank, and Lior Yanovski for many valuable discussions that led to this work.  They also thank Robert Burklund, Dustin Clausen, Peter Haine, Jeremy Hahn, Yonatan Harpaz,  and Hadrian Heine for helpful discussions and correspondences, and Robert Burklund, Dustin Clausen, Shay Ben Moshe, Tomer Schlank, and Lior Yanovski for comments on a draft.  Finally, they are grateful to the anonymous referees for their careful reading of this paper and wealth of clarifying comments and corrections.  
The first author was partially supported by the Adams Fellowship of the Israeli Academy of Science, and by the Danish National Research Foundation through the Copenhagen Centre for Geometry and Topology (DNRF151).
The second author was supported in part by NSF grant DMS-2002029.

\section{Higher commutative monoids and higher semiadditivity}\label{sect:prelim}



We saw in the introduction (cf. \Cref{intro:highercomm}) that the objects in a higher semiadditive $\infty$-category admit additional algebraic structure: in particular, for any $X$ in an $m$-semiadditive $\infty$-category and any $m$-finite space $A$, there is a natural map
\[
\push{A}: X^A \to X
\]
which can be thought of as implementing ``$A$-fold addition,'' by analogy to the case when $A$ is a finite set.  The $\push{A}$ maps, as $A$ ranges over $m$-finite spaces, fit together into an algebraic structure known as an $m$-commutative monoid, first formalized by Harpaz \cite{HarpazSpans}.  In this section, we define $m$-commutative monoids and their variants, and review some of the theory of higher semiadditive $\infty$-categories from this perspective.

\subsubsection{Higher commutative monoids}
For our applications to Grothendieck-Witt theory, it will be convenient to work in a $p$-typical setting, where all homotopy groups are $p$-groups:

\begin{defn}
We say that a $\pi$-finite space is $\tdef{$p$-typical}$ if all its homotopy groups are of $p$-power order. We denote by $\mdef{\Sfin{m}{p}}$ the $\infty$-category of  $m$-truncated $p$-typical $\pi$-finite spaces: that is, spaces $A$ with finitely many connected components, such that $\pi_i(A,a)$ is a finite $p$-group for all $i>0$ and $a\in A$, and vanishes for $i>m$. 
\end{defn}

The notion of a $p$-typical $m$-commutative monoid is defined using the $\infty$-category of spans $\Span(\Sfin{m}{p})$ \cite[\S 5]{Barwick}.  Informally, its objects are $m$-truncated $p$-typical $\pi$-finite spaces, and a morphism between two such spaces $A$ and $B$ is a third space $C\in \Sfin{m}{p}$ together with a pair of maps $A \leftarrow C \rightarrow B$. 

\begin{rem}
When $m=0$, $\Span(\Sfin{m}{p})$ is simply the $(2,1)$-category $\Span(\Fin)$ of spans of finite sets. 
\end{rem} 

It is a classical result that for a category $\cC$ with products, a commutative monoid in $\cC$ is a product preserving functor from $\Span(\Fin)^{\op}$ to $\cC$\footnote{See \cite{Cranch} for a detailed explanation and a generalization of this to the $\infty$-categorical setting.}.  By analogy, we have:

\begin{defn}[Definition 5.10 \cite{HarpazSpans}]\label{defn:cmon}
Let $\cC \in \Cat_\infty$. A $p$-typical \tdef{$m$-commutative pre-monoid} in $\cC$ is a functor $M\colon \Span(\Sfin{m}{p})^\op \to \cC$. We say that $M$ is a $p$-typical \tdef{$m$-commutative monoid} if it additionally satisfies the following Segal-type condition: 
\begin{itemize}
\item[($\ast$)] For every $A\in \Sfin{m}{p}$, the \mdef{Segal map} 
\[
\rho_A: M(A) \to \lim_A M(\pt)
\]
induced by the maps $
\{a^*\colon M(A) \to M(\pt)\}_{a\in A}
$, where $a\colon \pt \to A$ are the point embeddings, is an equivalence.
\end{itemize}
\end{defn}

We denote by 
\[
\mdef{\CMon{m}{p}(\cC)}\subseteq \mdef{\PMon{m}{p}(\cC)} := \Fun(\Span(\Sfin{m}{p})^\op,\cC)
\] the $\infty$-category of $p$-typical $m$-commutative pre-monoids and the full subcategory of $p$-typical $m$-commutative monoids in it.
If $\cC$ admits  $\Sfin{m}{p}$-limits, then, as in \cite[Proposition 5.14]{HarpazSpans}, we have 
\[
\CMon{m}{p}(\cC)\simeq \Fun^{\Sfin{m}{p}}(\Span(\Sfin{m}{p})^\op,\cC). 
\]
Here, $\Fun^{\Sfin{m}{p}}$ denotes functors which preserve limits indexed by spaces in $\Sfin{m}{p}$.

\begin{rem}

Informally, a $p$-typical $m$-commutative pre-monoid consists of restriction maps 
\[
    \mdef{f^*}\colon M(B)\to M(A)
\] 
for $f\colon A\to B \in \Sfin{m}{p}$, and push-forward, or integration, maps
\[
    \mdef{\push{f}} \colon M(A) \to M(B).
\]
The functoriality in spans encodes Fubini and base-change type compatibilities of the restriction and integration maps, as in \cite[Proposition 3.1.13 \& Corollary 3.1.14]{TeleAmbi}.  
In the special case where $M$ is a $p$-typical $m$-commutative monoid, and denoting $M(\pt)$ again by $M$, we can view the above as an integration of \emph{$M$-valued functions} on $A$ along the fibers of the map $f\colon A\to B$.
\end{rem}

\begin{rem}
We can view a $1$-commutative pre-monoid $M$ as a ($p$-typical) \emph{global equivariant object} in $\cC$ (in the sense of \cite{schwede2018global}) endowed with extra ``exotic transfer,'' or ``deflation'' maps along morphisms of groupoids with non-discrete fibers. 
In this language, the property of being a ($p$-typical) $m$-commutative monoid corresponds to the property of being Borel-complete. 
\end{rem}

\subsubsection{From higher commutative monoids to higher semiadditivity} \label{subsub:harpazwork}
By work of Harpaz, the notion of a higher semiadditive $\infty$-category can be formulated in terms of higher commutative monoids.  We review this formulation here, leading up to the key fact (\Cref{prop:highercommsemiadd}(3)) that objects in a ($p$-typical) $m$-semiadditive $\infty$-category admit a unique structure of a ($p$-typical) $m$-commutative monoid.  

Recall that $\cC$ is ($p$-typically) $m$-semiadditive if, for every map of spaces $f\colon A\to B$ with ($p$-typical) $m$-truncated $\pi$-finite fibers, the norm map \[
\Nm_f \colon f_! \to f_* 
\]
between functors $\cC^A\to \cC^B$
(which is defined inductively) 
is an isomorphism (\cite[\S 4]{HL},\cite[Definition 3.1.1]{AmbiHeight}).   For us, it will be convenient to use an alternative definition due to Yonatan Harpaz, that we shall now present. 

\begin{notation}
Denote by $\mdef{\Cat_{\Sfin{m}{p}}}$ the (non-full) subcategory of $\Cat_\infty$ consisting of small $\infty$-categories which admit $\Sfin{m}{p}$-colimits, and functors preserving $\Sfin{m}{p}$-colimits between them. 
We denote by $\mdef{\Catsa{m}{p}}\subseteq \Cat_{\Sfin{m}{p}}$ the full subcategory spanned by the  
$p$-typically $m$-semiadditive $\infty$-categories. 
\end{notation}

We will have to consider also the versions of these consisting of large $\infty$-categories. We will denote by 
$$\widehat{\Catsa{m}{p}} \subseteq \widehat{\Cat_{\Sfin{m}{p}}} \subseteq  \widehat{\Cat_{\infty}}$$
the $\infty$-categories of large $p$-typically $m$-semiadditive, large with $\Sfin{m}{p}$-colimits,  and large $\infty$-categories respectively. 

The $\infty$-category $\Cat_{\Sfin{m}{p}}$ admits a canonical symmetric monoidal structure via the Lurie tensor product (see \cite[Corollary 4.8.1.4]{HA}) which we denote by $\otimes$, with respect to which $\Sfin{m}{p}$ is the unit.  In this setting, Harpaz has given the following beautiful characterization of higher semiadditivity:

\begin{prop}\label{prop:highercommsemiadd} 
\hfill
\begin{enumerate}
\item The $\infty$-category $\Span(\Sfin{m}{p})$ is an idempotent algebra in $\Cat_{\Sfin{m}{p}}$.

\item The modules over $\Span(\Sfin{m}{p})$ in $\Cat_{\Sfin{m}{p}}$ are exactly the $p$-typically $m$-semiadditive $\infty$-categories; that is,  we have
\[
\Mod_{\Span(\Sfin{m}{p})}(\Cat_{\Sfin{m}{p}})
\simeq \Catsa{m}{p} \subseteq \Cat_{\Sfin{m}{p}}.\]

\item For a $p$-typically $m$-semiadditive $\infty$-category $\cC \in \Catsa{m}{p} $, the forgetful functor $\CMon{m}{p}(\cC) \to \cC$ is an equivalence.  
\end{enumerate}
\end{prop}

The non-$p$-typical case is discussed in \cite[\S 5.1 and Corollary 5.15]{HarpazSpans} and the proof for the $p$-typical case is identical\footnote{The necessary properties of the collection of $p$-typical $m$-finite spaces are that, like the collection of $m$-finite spaces, they are closed under finite limits and extensions.}.  Roughly, the first two parts of the proposition articulate a sense in which the $\infty$-category $$
\Span(\Sfin{m}{p}) \in \calg(\Cat_{\Sfin{m}{p}})$$
is universal among $p$-typically $m$-semiadditive $\infty$-categories. 

The third part will be most critical to our arguments: it asserts that every object of $\cC$ admits a unique structure of a $p$-typical $m$-commutative monoid in $\cC$.  We shall use the following notation for this unique structure: 



\begin{defn}\label{defn:canhighercm}
Let $\cC \in \Catsa{m}{p}$. For   $X\in\cC$ we denote by $\mdef{X^{(-)}}\in \CMon{m}{p}(\cC)$ the unique $p$-typical $m$-commutative monoid in $\cC$ whose underlying object is $X.$
\end{defn}

\begin{example}\label{ex:sk1cmon}
    The $\infty$-categories $\Sp_{K(n)}$ and $\Sp_{T(n)}$ are $\infty$-semiadditive (i.e., $m$-semiadditive for every $m\geq 0$) \cite{HL, TeleAmbi}. We obtain higher commutative monoids
\[
    \Sph_{K(n)}^{(-)}\in \CMon{m}{p}(\Sp_{K(n)}), \hspace{10pt} \Sph_{T(n)}^{(-)}\in \CMon{m}{p}(\Sp_{T(n)}).
\]
    This paper is dedicated to studying these higher commutative monoids in the case $n=1$ (in which case they are same because the telescope conjecture holds at height 1, see, e.g., \cite[Theorem 4.1]{Bousfield}). 
\end{example}

\subsubsection{Cardinalities}

We saw in the introduction that if $M$ is an $m$-commutative monoid and $A$ is an $m$-finite space, then there is a natural endomorphism $|A|_M :M \to M$ known as the cardinality of $|A|$ (at $M$).  In fact, one can give this definition more generally for a $p$-typical $m$-commutative pre-monoid:

\begin{defn}(see \cite[Definition 2.1.5]{AmbiHeight})\label{defn:card}
    Let $\cC\in \Cat_\infty$ and let $M\in \PMon{m}{p}(\cC)$. For $A\in \Sfin{m}{p}$ we define the map  
\[
    \mdef{|A|}_M\colon M(\pt) \oto{\pi^*} M(A) \oto{\push{\pi}} M(\pt),
\] 
where $\pi\colon A\to \pt$ is the terminal map.  In other words, $|A|_M$ is the image of $(\pt \xleftarrow{\pi} A\xrightarrow{\pi} \pt)$.  
We refer to $|A|_M$ as the \tdef{cardinality} of $A$ at $M$.
\end{defn}
The cardinalities of $A$ on the various $M$'s assemble to a natural endomorphism of the functor $\ev_\pt \colon \PMon{m}{p}(\cC)\to \cC$ which evaluates at $\pt \in \Span(\Sfin{m}{p})$. We shall denote this natural transformation simply by $|A|$.
In particular, if $\cC$ is $p$-typically $m$-semiadditive, we can restrict $|A|$ to a natural transformation of the identity functor of $\cC\simeq \CMon{m}{p}(\cC)$. These natural transformations are studied systematically in \cite{AmbiHeight} -- in particular, one can show (see \cite[Remark 2.1.10(2)]{AmbiHeight}) that if $\cC$ is further \emph{symmetric monoidal} and its tensor product distributes over $\Sfin{m}{p}$-colimits, then $|A|$, as a natural transformation of the identity functor of $\cC$, is given by multiplication with the element $|A|_{\one_\cC}\in \pi_0\one_\cC$.

\section{Higher semiadditive Grothendieck-Witt theory} 
\label{sec:higher_QG}


Let $\Sph_{K(1)}$ denote the $K(1)$-local sphere at the prime $2$. Fix a prime $\ell$ congruent to $3$ or $5$ modulo $8$. 
The main construction of this paper is a model for the $2$-typical $1$-commutative
monoid 
\[
\Sph_{K(1)}^{(-)} \in \CMon{1}{2}(\Sp_{K(1)})\simeq \Sp_{K(1)}
\] 
as the $K(1)$-localization of a $2$-typical $1$-commutative pre-monoid $\GW^{(-)}(\FF_\ell) \in \PMon{1}{2}(\Sp)$, whose underlying spectrum is the \emph{Grothendieck-Witt} spectrum $\GW(\FF_\ell)$. 

Recall that for a commutative ring $R$, the connective spectrum $\GW(R)$ is the group completion of the commutative monoid of symmetric non-degenerate bilinear forms over projective $R$-modules. We refer the reader to \cite[\S 4]{nine-author2} for the precise definition and an extensive discussion of Grothendieck-Witt theory.  

\begin{rem}
We remark that we will exclusively consider symmetric bilinear forms over rings in which $2$ is invertible, where symmetric bilinear and quadratic forms are interchangeable.  
\end{rem}





The construction $R\mapsto \GW(R)$ depends only the symmetric monoidal $\infty$-category of $R$-modules. More generally, if $\cC$ is a semiadditive symmetric monoidal $\infty$-category in which the tensor product distributes over coproducts, we can define the connective spectrum $\GW(\cC)$ as the group completion of the monoid of dualizable objects in $\cC$ endowed with a non-degenerate symmetric bilinear form. 

In this section, we enhance this construction to take higher semiadditivity into account: in particular, we show that if $\cC$ is a $p$-typical $m$-semiadditively symmetric monoidal $\infty$-category (i.e.,  $\cC \in \calg(\Catsa{m}{p})$), then $\GW(\cC)$ naturally acquires the structure of a $p$-typical $m$-commutative pre-monoid \[
\GW^{(-)}(\cC): \Span(\Sfin{m}{p})^{\op} \to \Sp.
\]
Informally, the construction consists of the following data: for $A\in \Sfin{m}{p}$, we set
\[
\GW^A(\cC) := \GW(\cC^A).
\] 
Then, for every map $f\colon A\to B\in \Sfin{m}{p}$, we provide maps 
\[
f^*\colon \GW^B(\cC) \to \GW^A(\cC)
\]
and
\[
\push{f} \colon \GW^A(\cC) \to \GW^B(\cC).
\]
Ignoring the symmetric bilinear forms, these maps are given by pre-composition with $f$ and left Kan extension along $f$, respectively.
For a symmetric bilinear form $b\colon X\otimes X \to \one$ on a dualizable object $X\in \cC^B$, we set 
\[
f^*b\colon f^*X \otimes f^*X\simeq f^*(X\otimes X)\oto{f^*b} f^*\one_B\simeq \one_A. 
\]
To specify $\push{f}$ on a symmetric bilinear form $b$, it will be easiest to specify its mate. 

\begin{constr}
\label{def_push_quadratic_form}
    Let $f\colon A\to B\in \Sfin{m}{p}$,
    let $\cC\in \calg(\Catsa{m}{p})$, and let $(X,b)\in \QF(\cC^A)$.  Then $b$ determines a map 
    \[
    b^\vee \colon X\xrightarrow{\simeq} \Du(X),
    \]
which we call its \emph{mate}, where $\Du(X)$ denotes the (symmetric monoidal) dual of the dualizable object $X$.  Note that $b^\vee$ is an equivalence because $b$ is nondegenerate.  We then define the bilinear form 
\[
   \mdef{\push{f} b} \colon f_!X \otimes f_!X \to \one_B
\]
to be the mate of the composite 
\[
f_!(X) \oto{f_!b^\vee} f_!(\Du(X))\oto{\Nm_f} f_* (\Du(X)) \simeq \Du(f_!(X)). 
\]
\end{constr}

Although it is not clear from this definition, it is a consequence of our results that this procedure defines a \emph{symmetric} bilinear form. 
Our goal for the rest of this section is to show that these constructions, currently given at the level of individual maps in $\Span(\Sfin{m}{p})$, assemble into a functor $\GW^{(-)}\colon \Span(\Sfin{m}{p})^\op \to \Sp$.

\subsection{Semiadditive anti-involutions}

Let $\cC$ be a symmetric monoidal $\infty$-category.  Our goal in this section will be to construct a space $\QF(\cC)$ of non-degenerate symmetric bilinear forms in $\cC$, and to explain the compatibility of this construction with higher semiadditivity and the multiplicative structure on $\cC$.  


We saw above that the data of a symmetric bilinear form $b$ on a dualizable object $X\in \cC$ can be encoded via its mate 
$
b^\vee \colon X\to \Du(X).
$  
In \Cref{subsub:antiinvol}, we construct a categorical framework for organizing such data.  
In fact, we will work in a more general setting where $\Du$ is replaced by any \emph{anti-involution} $\inv : \cC \iso \cC^{\op}$ of $\cC$ (cf. \Cref{defn:antiinvol}). Given an anti-involution $\inv$ on $\cC$, we will construct its $\infty$-category of \emph{fixed points} $\Phi(\cC, \inv)$: informally, an object of $\Phi(\cC, \inv)$ is the data of an object $X \in \cC$ together with an equivalence $X \simeq \inv (X)$.  We then construct $\QF(\cC)$ in the above setting by taking $\inv$ to be the anti-involution given by formation of symmetric monoidal duals (\Cref{ex:QF_as_fixed_points}).

In \Cref{subsub:ambiinvol}, we study the interaction of this construction with semiadditivity.  In particular, we will see (\Cref{thm:semifixedpoint}) that if $\cC$ is additionally $p$-typically $m$-semiadditive, then its $\infty$-category of fixed points $\Phi(\cC, \inv)$ acquires the structure of a $p$-typical $m$-commutative monoid in $\Cat_{\infty}$.  

\subsubsection{Anti-involutions and their fixed points}
\label{subsub:antiinvol}
Given an $\infty$-category $\cC$, one can pass to its opposite $\infty$-category $\cC^{\op}$.  
The assignment $\cC \mapsto \cC^{\op}$ defines an action of $C_2$ on the $\infty$-category $\Cat_{\infty}$.  

\begin{notation}
Hereafter, we regard $\Cat_{\infty}$ as equipped with this $C_2$ action; for instance, $\Cat_{\infty}^{hC_2}$ will denote homotopy fixed points with respect to this action.  We will also have reason to consider the $\infty$-category of functors from $BC_2$ to $\Cat_{\infty}$, and we will denote this functor category by $\Cat_{\infty}^{BC_2}$.  
\end{notation}

\begin{defn}\label{defn:antiinvol}
An \tdef{$\infty$-category with anti-involution} is an object of $\Cat_\infty^{hC_2}$. For an $\infty$-category $\cC$, an anti-involution on $\cC$ is a refinement of $\cC\in \Cat_{\infty}$ to an object of $\Cat_\infty^{hC_2}$.  
\end{defn}

Concretely, an anti-involution on $\cC$ is an anti-equivalence $\inv\colon \cC \iso \cC^\op$ together with a natural isomorphism $\inv^2 \iso \Id$ as well as higher coherence data.  We will primarily be interested in the following prototypical example:

\begin{example}
\label{duality_anti_involution}
Let $\cC$ be a symmetric monoidal $\infty$-category.
Then the full subcategory $\cC^\dbl \subseteq \cC$ spanned by the dualizable objects admits a canonical anti-involution $\Du\colon \cC^\dbl \iso (\cC^\dbl)^\op,$ given by taking the symmetric monoidal dual. This construction determines a functor 
\[
\mdef{(-)^\dbl} \colon \calg(\Cat_\infty) \to \Cat_\infty^{hC_2} 
\]
(cf. \cite[Theorem 5.11]{HeineLopezSpitzweck}).
\end{example}

We shall now construct the fixed points of an $\infty$-category with anti-involution. Recall that given an $\infty$-category $\cC$, one can consider its \emph{twisted arrow $\infty$-category} $\Tw (\cC)$: informally, the objects of $\Tw(\cC)$ are given by arrows $(X\to Y)$ in $\cC$ and the morphisms from $(X\to Y)$ to $(X'\to Y')$ can be thought of as commutative squares
\[
\begin{tikzcd}
X \arrow[r] \arrow[d] & Y \\
X'\arrow[r] & Y'\arrow[u].
\end{tikzcd}
\]
One can construct $\Tw$ more formally as follows: let $\Delta$ denote the category of non-empty linearly ordered finite sets.  The complete Segal space construction of Rezk \cite{Rezk} determines a fully faithful embedding $\Cat_\infty \into \Spc^{\Delta^{\op}}$ as presheaves of spaces on $\Delta$ satisfying a certain ``complete Segal'' condition \cite[Corollary 4.3.16]{LurGood}. 

The $\infty$-category $\Delta$ has an involution 
\[
(-)^\rev \colon \Delta \to \Delta,
\]
which reverses the order of a totally ordered set, and precomposition with $(-)^\rev$ induces the $C_2$ action $(-)^\op$ of $\Cat_\infty$ (cf. \cite[\S 1.2.1]{HTT}).  Additionally, the category $\Delta$ admits a monoidal structure by concatenation of linearly ordered sets, which we will denote by $\star$.  One can then consider the functor 
\[
\Tw\colon \Spc^{\Delta^\op} \to
\Spc^{\Delta^\op}
\] 
given by
\[
\Tw(\cC)(I)=\cC(I\star I^\rev).
\]
This functor takes complete Segal spaces to complete Segal spaces, so it restricts to an endofunctor of $\Cat_\infty$ (c.f.  \cite[Proposition A.2.3]{ShiftedCo}): we denote this endofunctor also by 
\[\Tw\colon \Cat_{\infty} \to \Cat_{\infty}.\]

When the $\infty$-category $\cC$ additionally comes equipped with an anti-involution $\inv$, the $\infty$-category $\Tw(\cC)$ acquires a natural $C_2$ action given on objects by 
\[
(X\oto{f} Y) \mapsto (\inv (Y) \oto{\inv (f)} \inv (X)).
\]
Moreover, this procedure is functorial:

\begin{lem}
The functor $\Tw \colon  \Cat_\infty\to \Cat_\infty$ lifts to a functor 
\[
\Tw\colon \Cat_\infty^{hC_2}\to \Cat_\infty^{BC_2}
\]
together with a $C_2$-equivariant natural transformation $\Tw(\cC) \to \cC^{\op}\times \cC$.  Here, the $C_2$ action on the target is given by $(X,Y) \mapsto (\inv (Y), \inv(X))$.  
\end{lem}
\begin{proof}

To demonstrate the lemma, we note that there is a natural isomorphism $(I\star I^\rev)^\rev \cong I\star I^\rev$ in the ordinary category $\Delta$ and hence $\Tw$ intertwines the involution $(-)^\op$ of $\Cat_{\infty}$ with the trivial involution of $\Cat_\infty$. Consequently, it induces a functor between the corresponding fixed points
\[
\Tw\colon \Cat_\infty^{hC_2}\to \Cat_\infty^{BC_2}.\]
Moreover,  the forgetful map $\Tw(\cC) \to \cC^\op \times \cC$ is induced from the natural embeddings $I\into I\star I^\rev$ and $I^\rev \into I\star I^\rev$, and so it is visibly equivariant.  
\end{proof}

Let $(\cC,\inv) \in \Cat_\infty^{hC_2}$ be an $\infty$-category with an anti-involution.  Then the above lemma produces a natural $C_2$-equivariant functor \[
\Tw(\cC) \to \cC \times \cC^\op.
\]
There is also a natural functor
\[
\cC \oto{(\Id,\inv)} \cC \times \cC^\op 
\] 
which is equivariant when the source, $\cC$, is endowed with the \emph{trivial} $C_2$-action. 

\begin{defn}
\label{fixed_points_anti_involution}
We define the functor 
\[
    \mdef{\WInvF} \colon  \Cat_{\infty}^{hC_2}\to \Cat_\infty^{BC_2}
\]
by the formula
\[
    \WInvF(\cC,\inv)= \cC\times_{\cC \times \cC^\op} \Tw(\cC), 
\]
so that an object of $\WInvF(\cC,\inv)$ is a morphism $\alpha \colon X\to \inv(X)$ in $\cC$, and a morphism between $(X,\alpha)$ and $(Y,\beta)$ is a commutative diagram of the form: 
\[
\xymatrix{
X \ar^f[r]\ar^\alpha[d]       & Y \ar^\beta[d] \\ 
\inv(X) & \inv(Y) \ar^{\inv(f)}[l].
}
\]
We also define the functor
\[
\mdef{\InvF} \colon \Cat_{\infty}^{hC_2}\to \Cat_\infty^{BC_2}
\]
to be the full sub-functor
$
\InvF(\cC,\inv) \subseteq  
\WInvF(\cC,\inv)
$
spanned by the isomorphisms $
X\iso \inv(X).$  
\end{defn}
We may think of $\WInvF(\cC,\inv)$ and $\InvF(\cC,\inv)$ as the ``lax fixed points'' and ``fixed points,'' respectively, of $\inv$ acting on $\cC$.  While we are primarily interested in $\InvF$, it will be useful to treat these two functors in parallel. 


\begin{rem}
For the purposes of this paper, it is enough to consider the composition
\begin{equation} \label{simpler_Phi}
\Cat_\infty^{hC_2}\oto{\InvF} \Cat_\infty^{BC_2} \oto{(-)^\simeq} \Spc^{BC_2}
\end{equation}
where $(-)^{\simeq}$ denotes the maximal subgroupoid functor.  
This composition has a simpler description\footnote{We are grateful to the anonymous referee for pointing out this simplification.} which does not involve the construction $\InvF$. 
Namely, the non-trivial action of $C_2$ on $\Cat_\infty$ (given by $(-)^{\op}$) restricts to a trivial action of $C_2$ on $\Spc$. Hence, the functor $(-)^\simeq \colon \Cat_\infty\to \Spc$, being right adjoint to the inclusion $\Spc \into \Cat_\infty$, promotes to a $C_2$-equivariant functor, where the action on the target is trivial. Taking $C_2$-fixed points, we get a functor $\Cat_\infty^{hC_2} \to \Spc^{BC_2}$. Unwinding the definitions, this functor is exactly the composition (\ref{simpler_Phi}).  Nevertheless, we shall discuss the functor $\InvF$ as it constructs a \emph{category} of nondegenerate symmetric bilinear forms, and not just a space.
\end{rem}

Using this language, we can now define the space of objects equipped with a symmetric bilinear form in $\cC$ as follows. 

\begin{defn}
\label{ex:QF_as_fixed_points}
For a symmetric monoidal $\infty$-category $\cC\in \calg(\Cat_\infty)$, we define the space $\QF(\cC)$ by the formula 
\[
\QF(\cC)\simeq (\InvF(\cC^\dbl,\Du)^{hC_2})^{\simeq},
\]
\end{defn}
where $(-)^{\simeq}$ refers to the maximal subgroupoid. By definition, an object of $\InvF(\cC^\dbl,\Du)$ is a dualizable object $X\in \cC$ together with an isomorphism $X\iso \Du(X)$, which corresponds to a non-degenerate bilinear map $X\otimes X \to \one$. The group $C_2$ acts by swapping the two tensor factors, and so its fixed points are symmetric bilinear maps.

\subsubsection{Anti-involutions and semi-additivity}\label{subsub:ambiinvol}

Our aim now is to promote $\InvF(\cC,\inv)$ and $\WInvF(\cC,\inv)$
to $p$-typical $m$-commutative monoids when $\cC$ is $p$-typically $m$-semiadditive. 
This will essentially follow from the facts that these functors are compatible with the formation of $p$-typical $m$-commutative monoids (\Cref{fixed_anti_inv_limits}), and that every object of $(\Catsa{m}{p})^{hC_2}$ admits a canonical such structure (\Cref{semi_add_with_anti_semi_add}). 

To see that $\InvF$ and $\WInvF$ induce functors on $m$-commutative monoids, we show:
\begin{lem}
\label{fixed_anti_inv_limits}
The functors $\InvF\colon \Cat_\infty^{hC_2} \to \Cat_\infty^{BC_2}$ and $\WInvF\colon \Cat_\infty^{hC_2} \to \Cat_\infty^{BC_2}$ are limit preserving.
\end{lem}

\begin{proof}
First, since the evaluation at the base-point functor $\Cat_\infty^{BC_2}\to \Cat_\infty$ is conservative and limit preserving, it suffices to verify the claim non-equivariantly. 

To see that $\WInvF$ is limit preserving, it suffices to see that each of the three functors in the pullback
\[
\WInvF(\cC,\inv) = \cC \times_{\cC\times \cC^\op} \Tw(\cC)
\]
preserves limits.  
The functors $\cC \mapsto \cC$ and $\cC \mapsto \cC \times \cC^\op$ are clearly limit preserving. The functor $\cC \mapsto \Tw(\cC)$ is limit preserving since it is given by pre-composition with a map $\Delta \to \Delta$ on $\Spc^{\Delta^\op}$, and the fully faithful embedding $\Cat_\infty \to \Spc^{\Delta^\op}$ is conservative and limit preserving.

To deduce that $\InvF$ is limit preserving, it is enough to show that, for every (small) $\infty$-category $I$ and every $I$-shaped diagram $\{(\cC_a,\inv_a)\}_{a\in I}$ in $\Cat_\infty^{hC_2}$, the equivalence 
\[
\WInvF(\invlim_{a\in I} (\cC_a,\inv_a)) \simeq \invlim_{a\in I} \WInvF((\cC_a,\inv_a))
\] 
carries the full subcategory $\InvF(\invlim_{a\in I}(\cC_a,\inv_a))\subseteq \WInvF(\invlim_{a\in I}(\cC_a,\inv_a))$ to 
$\invlim_{a\in I} \InvF(\cC_a,\inv_a)$. Unwinding the definitions, this reduces to the fact that for every diagram $I\to \Cat_\infty^{hC_2}$, a map
\[
\{X_a\}_{a\in I} \oto{\{\phi_a\}_{a\in I}} \{\inv_a(X_a)\}_{a\in I} \qin \invlim_{a\in I} \cC_a 
\]
is an isomorphism if and only if each $\phi_a$ is an isomorphism. 

\end{proof}

We now turn to the second ingredient in the refinement of $\InvF$: the higher commutative monoid structures on those $(\cC,\inv)\in \Cat_\infty^{hC_2}$ for which $\cC$ is $p$-typically $m$-semiadditive. 
In fact, we shall package such pairs in a $p$-typically $m$-semiadditive $\infty$-category. 
\begin{prop}\label{semiadd_op}
The involution $(-)^\op$ of $\Cat_\infty$ restricts to an involution of the (non-full) subcategory $\Catsa{m}{p}$ of $\Cat_\infty$.
\end{prop}

\begin{proof}
Since $\Catsa{m}{p}$ is a subcategory of $\Cat_\infty$, it suffices to see that 
\begin{enumerate}
    \item If $\cC\in \Catsa{m}{p}$, then $\cC^\op \in \Catsa{m}{p}$.  This follows from \cite[Proposition 2.1.4(1)]{AmbiHeight}.
    \item A functor $F \colon \cC_0 \to \cC_1$ in $\Catsa{m}{p}$ preserves $\Sfin{m}{p}$-colimits if and only if  $F^\op\colon \cC_0^\op \to \cC_1^\op$ does.  This is because a functor between $p$-typically $m$-semiadditive $\infty$-categories preserves $\Sfin{m}{p}$-limits if and only if it preserves $\Sfin{m}{p}$-colimits (\cite[Proposition 2.1.8]{AmbiHeight}).
\end{enumerate}
\end{proof}

In view of this result, the $C_2$-action on $\Cat_{\infty}$ by $(-)^{\op}$ restricts to an action on the subcategory $\Catsa{m}{p}\subset \Cat_{\infty}$, and we can form the $\infty$-category $(\Catsa{m}{p})^{hC_2}$ of $p$-typically $m$-semiadditive $\infty$-categories with anti-involution.  We will now show that this category is itself higher semiadditive.  

\begin{lem}\label{lem:lim_of_semiadd_cats}
The inclusion of the (non-full) subcategory $\widehat{\Catsa{m}{p}} \subset \widehat{\Cat_{\infty}}$ preserves limits. 
\end{lem}
\begin{proof}
The claim follows from the fact that this inclusion admits a left adjoint; the analogous statement in the non-$p$-typical case, and for small $\infty$-categories, is shown in \cite[Theorem  5.28(2)]{HarpazSpans}, and the $p$-typical case for large $\infty$-categories follows similarly. 


\end{proof}

\begin{cor}
\label{semi_add_with_anti_semi_add}
The $\infty$-category $(\Catsa{m}{p})^{hC_2}$ is $p$-typically $m$-semiadditive and the canonical functor $(\Catsa{m}{p})^{hC_2}\to \Cat_\infty^{hC_2}$ is limit preserving. 
\end{cor}
  
\begin{proof}
By (the straightforward $p$-typical analog of) \cite[Proposition 2.2.11]{AmbiHeight}, the $\infty$-category $\Catsa{m}{p}$ is $p$-typically $m$-semiadditive.  By \Cref{semiadd_op}, the involution $(-)^{\op}$ of $\Cat_\infty$ restricts to an involution of $\Catsa{m}{p}$, so that we obtain a functor $BC_2\to \widehat{\Catsa{m}{p}}$ classifying this involution of $\Catsa{m}{p}\in \widehat{\Catsa{m}{p}}$. By \Cref{lem:lim_of_semiadd_cats} applied to this functor we deduce that $(\Catsa{m}{p})^{hC_2}$ is $p$-typically $m$-semiadditive as well.  

It remains to show that the functor $(\Catsa{m}{p})^{hC_2} \to \Cat_\infty^{hC_2}$ is limit preserving.  This is because the (non-fully faithful) inclusion $\Catsa{m}{p} \to \Cat_{\infty}$ is limit preserving (by the small category analogue of \Cref{lem:lim_of_semiadd_cats}) and limit-preserving functors are closed under limits in $\widehat{\Cat_{\infty}}$.
\end{proof}

Let us illustrate this corollary more concretely:
Let $\cC$ be a $p$-typically $m$-semiadditive $\infty$-category.  Then $\cC$ acquires a canonical $p$-typical $m$-commutative monoid structure, which provides for any map $f:A\to B$ in $\Sfin{m}{p}$ an integration functor $\push{f} = f_! :\cC^A \to \cC^B$ (given by left Kan extension).  Now suppose $\cC$ is equipped with an anti-involution $\inv$, so $(\cC, \inv) \in (\Catsa{m}{p})^{hC_2}$.  Then one can build a diagram
\begin{equation}\label{eqn:invhighercomm}
\begin{tikzcd}
\cC^A \arrow[r,"f_!"] \arrow[d, "\inv^A"]& \cC^B \arrow[d, "\inv^B"] \\
(\cC^A)^{\op} \arrow[r,"(f_!)^{\op}"] & (\cC^B)^{\op}.\\
\end{tikzcd}
\end{equation}
 \Cref{semi_add_with_anti_semi_add} implies that since $(\Catsa{m}{p})^{hC_2}$ is $p$-typically $m$-semiadditive, $(\cC, \inv)$ acquires a canonical $m$-commutative monoid structure in $(\Catsa{m}{p})^{hC_2}$.  Part of such data is a witness to the commutativity of (\ref{eqn:invhighercomm}): unwinding the constructions, this homotopy can be expressed concretely by the formula
\begin{equation} \label{eqn:semiaddinv}
(f_!)^{\op} \inv^A \simeq (f_*)^{\op} \inv^A \simeq  \inv^B f_! 
\end{equation}
where the first equivalence  comes from the norm equivalence $f_! \simeq f_*$ arising from the higher semiadditivity of $\cC$, and the second comes from $\inv$ being an anti-involution.

We now define a higher semiadditive version of the functors $\InvF$ and $\WInvF$.

\begin{thm}\label{thm:semifixedpoint}
The functors $\InvF\colon (\Catsa{m}{p})^{hC_2}\to \Cat_\infty^{BC_2}$ and $\WInvF\colon (\Catsa{m}{p})^{hC_2}\to \Cat_\infty^{BC_2}$ 
admit refinements to functors 
\[
 (\Catsa{m}{p})^{hC_2}\to \CMon{m}{p}(\Cat_\infty)^{BC_2}.
\] 
In other words, for $(\cC,\inv)$ a $p$-typically $m$-semiadditive $\infty$-category with anti-involution, the $\infty$-categories
$\InvF(\cC,\inv)$ and $\WInvF(\cC,\inv)$ are  canonically  $C_2$-equivariant $p$-typical $m$-commutative monoids in $\Cat_\infty$. 
\end{thm}

\begin{proof}
We prove the result for $\InvF$, the proof for $\WInvF$ being completely analogous.
Note that if $\cD_0$ and $\cD_1$ are $\infty$-categories which admit $\Sfin{m}{p}$-limits and $F\colon \cD_0\to \cD_1$ is a $\Sfin{m}{p}$-limit preserving functor, then post-composition with $F$ induces a functor 
\[
\CMon{m}{p}(\cD_0) \to \CMon{m}{p}(\cD_1).
\]
Indeed, the Segal condition involves only $\Sfin{m}{p}$-limits and hence $F$ preserves this condition. 
Consequently, by
\Cref{fixed_anti_inv_limits}, the functor $\InvF\colon \Cat_\infty^{hC_2}\to \Cat_\infty^{BC_2}$ induces a functor 
\[
\CMon{m}{p}(\Cat_\infty^{hC_2})\to \CMon{m}{p}(\Cat_\infty^{BC_2})\simeq 
\CMon{m}{p}(\Cat_\infty)^{BC_2}
.\]
Similarly, by the second part of \Cref{semi_add_with_anti_semi_add}, the functor $(\Catsa{m}{p})^{hC_2} \to \Cat_\infty^{hC_2}$ is limit-preserving and hence induces a functor \[
\CMon{m}{p}((\Catsa{m}{p})^{hC_2})\to \CMon{m}{p}(\Cat_\infty^{hC_2}).
\]
 Finally, by the first part of \Cref{semi_add_with_anti_semi_add}, $(\Catsa{m}{p})^{hC_2}$ is $p$-typically $m$-semidditive and hence we have a canonical equivalence
 \[
 (\Catsa{m}{p})^{hC_2} \iso \CMon{m}{p}((\Catsa{m}{p})^{hC_2})
 \]
 (see \Cref{prop:highercommsemiadd}).
 Composing these three functors we get the desired refinement. 
\end{proof}
 
By abuse of notation, we denote the resulting functors 
\[
(\Catsa{m}{p})^{hC_2} \to \CMon{m}{p}(\Cat_\infty)^{BC_2}
\]
again by $\InvF$ and $\WInvF$, so that if $\cC$ is $p$-typically $m$-semiadditive, then $\InvF(\cC,\inv)$ and $\WInvF(\cC,\inv)$ are higher commutative monoids in $\Cat_\infty^{BC_2}$. 

We end this section with an explicit description of this higher commutative monoid structure on the level of objects.


\begin{prop}
\label{formula_push_anti_inv_fixed}
Let 
$
(\cC,\inv)\in (\Catsa{m}{p})^{hC_2}$.  For every map $f:A\to B$ of $p$-typical $m$-finite spaces, the higher commutative monoid structure on $\WInvF(\cC,\inv)$ provides a functor 
\[
    \push{f}\colon \WInvF(\cC,\inv)^A \to \WInvF(\cC,\inv)^B.
\]
On an object $(\phi\colon X\to \inv(X)) \in \WInvF(\cC,\inv)^A$ (where $X\in \cC^A$ and $\inv$ is applied pointwise), this functor is given by the formula
\[
\push{f}(X\oto{\phi} \inv(X)) = (f_!X \oto{\psi} \inv(f_!X)) \qin \WInvF(\cC,\inv)^{B},
\]
where $\psi$ is the composite
\[
f_!X \oto{f_! \phi} f_!\inv(X) \simeq \inv(f_*X)\oto{\inv(\Nm_f)} \inv(f_! X). 
\]
\end{prop}
\begin{proof}

The higher commutative monoid structure of $\WInvF(\cC, \inv)$ is defined by applying the (limit preserving) functor given by 
\[
\WInvF(\cC, \inv) = \cC \times_{\cC \times \cC^{\op}} \Tw(\cC)
\]
to the higher commutative monoid $(\cC,\inv)\in \CMon{m}{p}((\Catsa{m}{p})^{hC_2})$.  Here, we emphasize that the functor $\cC \to \cC \times \cC^{\op}$ is given by $(\mathrm{id}, \inv)$.

Using this pullback, we find that the integration map $\push{f}$ on an object $(\phi\colon X\to \inv(X)) \in \WInvF(\cC,\inv)^A$ consists of the following collection of data:
\begin{enumerate}
    \item The integration of $X \in \cC^A$, which is $f_!X \in \cC^B$.  
    \item The integration of the (twisted) arrow $(X\xrightarrow{\phi} \inv(X))$ in $\cC^A$, which is the (twisted) arrow $(f_!X \xrightarrow{f_! \phi} f_!\inv(X))$ in $\cC^B$.
    \item The identification of their corresponding images in $\cC^B \times (\cC^B)^{\op}$.  Referring to the first two points, this amounts to identifying the pairs $(f_!X, \inv(f_!X))$ and $(f_!X, f_!\inv(X))$.  By definition, this uses the compatibility of $f_!$ with $\inv$ inherent in the higher commutative monoid structure of $(\cC,\inv)\in \CMon{m}{p}((\Catsa{m}{p})^{hC_2})$.  As in the discussion after \Cref{semi_add_with_anti_semi_add} (and in particular (\ref{eqn:semiaddinv})), this is given by the composite
    \[
    f_!\inv(X) \simeq \inv(f_* X) \xrightarrow{\inv(\Nm_f)} \inv(f_! X).\footnote{Note that $\inv$ is contravariant and hence flips the direction of the norm map.}
    \]
\end{enumerate}

Putting together this data, we obtain the claimed formula.  
\end{proof}

\begin{rem}
Since the fully faithful embedding $\InvF(\cC,\inv)\subseteq \WInvF(\cC,\inv)$ is compatible with the higher commutative monoid structure, the same formula applies for $\InvF$ instead of $\WInvF$.  
\end{rem}

\subsubsection{Lax symmetric monoidal structure on $\InvF$}

In the previous section, we defined a functor 
\[
\InvF\colon (\Catsa{m}{p})^{hC_2}\to \CMon{m}{p}(\Cat_\infty)^{BC_2}
\]
taking an $m$-semiadditive $\infty$-category $\cC$ with an anti-involution $\inv\colon \cC\iso \cC^\op$ to its $\infty$-category of fixed points $\InvF(\cC,\inv)$, considered as a ($C_2$-equivariant) higher commutative monoid in categories.  The purpose of this section is to endow $\InvF$ with multiplicative structure: more precisely, we will describe symmetric monoidal structures on the source and target of $\InvF$, and show that $\InvF$ is lax symmetric monoidal with respect to these structures. In \Cref{subsub:mult}, this will be used to endow the Grothendieck-Witt spectrum of a higher semiadditive symmetric monoidal $\infty$-category with a ring structure.

Recall that $\Catsa{m}{p}$ admits a symmetric monoidal structure coming from the Lurie tensor product (cf. \Cref{subsub:harpazwork}).  We first show that this lifts to a symmetric monoidal structure on $(\Catsa{m}{p})^{hC_2}$.
\begin{prop}\label{prop:hc2symmon}
The $C_2$-action on $\Catsa{m}{p}$ given by $\cC \mapsto \cC^{\op}$ lifts canonically to a symmetric monoidal action, that is, a $C_2$-action on $\Catsa{m}{p}$ considered as an object of $\calg(\widehat{\Cat_\infty})$.
\end{prop}

\begin{proof}
For this proof, we will use the standard notations from the theory of $\infty$-operads of \cite{HA}.  Let $\Cat_{\infty}^{\otimes}$ denote the $\infty$-operad corresponding to the Cartesian symmetric monoidal structure on $\Cat_{\infty}$.  Then, the $\infty$-operad $(\Catsa{m}{p})^{\otimes}$ corresponding to the symmetric monoidal $\infty$-category $\Catsa{m}{p}$ is given as the subcategory $(\Catsa{m}{p})^{\otimes}\subset \Cat_{\infty}^{\otimes}$ where (cf. \cite[Notation 4.8.1.2]{HA}):
\begin{enumerate}
    \item The objects are the sequences $(\cC_1, \cC_2, \cdots, \cC_k)$ where each $\cC_i \in \Catsa{m}{p}$.  
    \item A morphism $(\cC_1, \cC_2, \cdots, \cC_k) \to (\cD_1, \cD_2, \cdots , \cD_j)$ covering $\alpha: \langle k\rangle \to \langle j \rangle $  is in the subcategory $(\Catsa{m}{p})^{\otimes} \subset \Cat_{\infty}^{\otimes}$ iff each component functor
    \[
    \prod_{i\in \alpha^{-1}(\ell)} \cC_i \to \cD_\ell
    \]
    preserves $\Sfin{m}{p}$-colimits separately in each variable.  
\end{enumerate}

The $C_2$-action by $\op$ on $\Cat_{\infty}$ is symmetric monoidal, and thus induces a $C_2$-action on $\Cat_{\infty}^{\otimes}$.  To see that this restricts to a symmetric monoidal action on $\Catsa{m}{p}$, it suffices to check that:
\begin{enumerate}
    \item The action map $\op: \Cat_{\infty}^{\otimes} \to \Cat_{\infty}^{\otimes}$ preserves the subcategory $(\Catsa{m}{p})^{\otimes}$.  This is true on objects because the opposite of a $p$-typically $m$-semiadditive $\infty$-category is $p$-typically $m$-semiadditive.  It follows for morphisms because functors between $p$-typically $m$-semiadditive $\infty$-categories preserve $\Sfin{m}{p}$-colimits if and only if they preserve $\Sfin{m}{p}$-limits.  
    \item The resulting functor $\op: (\Catsa{m}{p})^{\otimes} \to (\Catsa{m}{p})^{\otimes}$ sends coCartesian arrows to coCartesian arrows.  This is immediate for the inert morphisms, and follows for the active morphisms again from the fact that functors between $p$-typically $m$-semiadditive $\infty$-categories preserve $\Sfin{m}{p}$-colimits if and only if they preserve $\Sfin{m}{p}$-limits.  
\end{enumerate} 
\end{proof}

As a consequence of this result, we can regard the fixed points $(\Catsa{m}{p})^{hC_2}$ as taken within the $\infty$-category of symmetric monoidal $\infty$-categories; this equips $(\Catsa{m}{p})^{hC_2}$ with a symmetric monoidal structure.

Next, we turn our attention to functor $\InvF$.  We start by considering a closely related functor:

\begin{lem}\label{lem:Philaxsym}
Let $(\Catsa{m}{p})^{hC_2}$ be given a symmetric monoidal structure via \Cref{prop:hc2symmon}, and regard $\Cat_\infty^{BC_2}$ as having the Cartesian symmetric monoidal structure.  Then there is a canonical lax symmetric monoidal structure on the composite
\[
(\Catsa{m}{p})^{hC_2} \xrightarrow{j} \Cat_{\infty}^{hC_2} \xrightarrow{\InvF} \Cat_{\infty}^{BC_2},
\]
where $j$ denotes the forgetful functor.
\end{lem}
\begin{proof}
This is because each functor in the composite admits a lax symmetric monoidal structure:
\begin{enumerate}
    \item  By the proof of \Cref{prop:hc2symmon}, the forgetful functor $\Catsa{m}{p} \to \Cat_{\infty}$ extends to a $BC_2$-family of lax symmetric monoidal functors.  Therefore, its limit, which is the functor $j$, acquires a lax symmetric monoidal structure. 
    \item The functor $\InvF$ preserves limits by \Cref{fixed_anti_inv_limits}, and both the source and the target are given the Cartesian symmetric monoidal structure; thus, $\InvF$ is symmetric monoidal.  
\end{enumerate}    
\end{proof}
Hence, to endow the fixed points functor $\InvF\colon (\Catsa{m}{p})^{hC_2}\to \CMon{m}{p}(\Cat_\infty)^{BC_2}$ with a lax symmetric monoidal structure, it remains to consider the effect of taking $m$-commutative monoids on lax symmetric monoidal functors. 

The $\infty$-category $\Span(\Sfin{m}{p})^\op$ admits a symmetric monoidal structure given by the Cartesian product in $\Sfin{m}{p}$ \cite[\S 2]{BarwickMackey2}. Hence, for every presentably symmetric monoidal $\infty$-category $\cC$, we can consider $\PMon{m}{p}(\cC)$ as a symmetric monoidal $\infty$-category via \emph{Day convolution} \cite{SaulDay}. 
\begin{prop}\label{cmon_pmon_symm_mon}
Let $\cC$ be a presentably symmetric monoidal $\infty$-category. 
Then $\CMon{m}{p}(\cC)$ is a symmetric monoidal localization of $\PMon{m}{p}(\cC)$, and in particular, it admits a unique symmetric monoidal structure for which the localization functor $\PMon{m}{p}(\cC)\to \CMon{m}{p}(\cC)$ is symmetric monoidal. Moreover, the equivalence $\CMon{m}{p}(\cC)\simeq \CMon{m}{p}(\Spc)\otimes \cC$ is a symmetric monoidal equivalence. 
\end{prop}

\begin{proof}
This follows from \cite[Theorem 4.27]{ShayTomer}.  
\end{proof}

Combining these results, we obtain
\begin{thm}\label{fixed_lax}
    The functor $\InvF\colon (\Catsa{m}{p})^{hC_2} \to \CMon{m}{p}(\Cat_\infty^{BC_2})$ is lax symmetric monoidal. 
\end{thm}

\begin{proof}
We can write $\InvF$ as a composite
\begin{align*}
(\Catsa{m}{p})^{hC_2} \iso \CMon{m}{p}((\Catsa{m}{p})^{hC_2}) &\into \PMon{m}{p}((\Catsa{m}{p})^{hC_2}) \\
&\oto{\InvF \circ (-)} \PMon{m}{p}(\Cat_\infty^{BC_2}) \oto{L} \CMon{m}{p}(\Cat_\infty^{BC_2})   ,
\end{align*}
where $L$ is the symmetric monoidal localization from \Cref{cmon_pmon_symm_mon} and $\InvF \circ (-)$ denotes postcomposition with $\InvF$ regarded as a functor $(\Catsa{m}{p})^{hC_2} \to \Cat_\infty^{BC_2}$.   We show, in order, that each map in this composite admits a lax symmetric monoidal structure. 

\begin{itemize}
\item The equivalence 
\[
    (\Catsa{m}{p})^{hC_2} \iso \CMon{m}{p}((\Catsa{m}{p})^{hC_2}) \simeq (\Catsa{m}{p})^{hC_2} \otimes \CMon{m}{p}(\Spc)
\] 
identifies with the tensor product of the identity functor of $(\Catsa{m}{p})^{hC_2}$ and the unit $\Spc \to \CMon{m}{p}(\Spc)$. Hence, it is a symmetric monoidal equivalence by the second part of \Cref{cmon_pmon_symm_mon}.

\item The inclusion $\CMon{m}{p}((\Catsa{m}{p})^{hC_2}) \into \PMon{m}{p}((\Catsa{m}{p})^{hC_2})$ is the right adjoint of the symmetric monoidal localization $L$, and hence is canonically lax symmetric monoidal. 

\item For symmetric monoidal $\infty$-categories $\cA$ and $\cB$, the Day convolution symmetric monoidal structure on $\Fun(\cA,\cB)$ (when it exists) is functorial in post-composing with lax symmetric monoidal functors. Hence, by \Cref{lem:Philaxsym},
the functor $\InvF \circ (-)$ is lax symmetric monoidal with respect to Day convolution. 

\item Finally, $L$ is a symmetric monoidal localization.
\end{itemize}
\end{proof}

\subsection{Grothendieck-Witt theory as a higher commutative monoid}

We now refine $\GW$ to a functor valued in $p$-typical $m$-commutative pre-monoids.  After some preliminaries on symmetric monoidal duality, we refine the functor $\QF$ from \Cref{ex:QF_as_fixed_points} to take $m$-semiadditively symmetric monoidal categories to $m$-commutative monoids (\Cref{{subsub:sbf}}).  Then, we apply group completion to obtain the semiadditive form of $\GW$ in \Cref{subsub:gpcomplete}, and discuss the multiplicative structure in \Cref{subsub:mult}.  Finally, in \Cref{subsub:gwfield}, we unwind the definition of the higher semiadditive structure on $\GW$ in the special case of vector spaces over a field.  

\subsubsection{Symmetric monoidal duality and semiadditivity}

Let $\cC$ be a symmetric monoidal $\infty$-category.  Then recall (\Cref{duality_anti_involution}) that one can extract the full subcategory $\cC^{\dbl}\subset \cC$ of dualizable objects.  This subcategory admits a canonical anti-involution by symmetric monoidal duality, and so the assignment $\cC \mapsto \cC^{\dbl}$ determines a functor
\[
(-)^{\dbl}\colon \calg(\Cat_\infty) \to \Cat_\infty^{hC_2}.
\]

We will need a higher semiadditive version of this functor:

\begin{prop}\label{prop:extractdbl}
The assignment $\cC \mapsto \cC^{\dbl}$ determines a functor
\[
(-)^{\dbl}\colon \calg(\Catsa{m}{p}) \to (\Catsa{m}{p})^{hC_2}.
\]
Moreover, this functor is lax symmetric monoidal with respect to the Lurie tensor product (cf. \Cref{prop:hc2symmon}).  
\end{prop}

The proof, which is assembled at the end of this section, will proceed by factoring $(-)^{\dbl}$ as a composite of several lax symmetric monoidal functors.  Let 
\[
\calg(\Catsa{m}{p})^{\rig} \subset \calg(\Catsa{m}{p})
\]
denote the full subcategory of \emph{rigid} $p$-typical $m$-semiadditively symmetric monoidal categories -- that is, those $\cC \in \calg(\Catsa{m}{p})$ satisfying the condition that every object $X\in \cC$ is dualizable.  

Before we proceed, we shall need the following general lemma about the Lurie tensor product.  

\begin{lem}\label{lem:tensor-gen}
Let $\mathcal{K}$ be a set of simplicial sets, let $\Cat_{\mathcal{K}}$ denote the $\infty$-category of $\infty$-categories which admit $\mathcal{K}$-shaped colimits (with morphisms functors preserving those colimits), and give $\Cat_{\mathcal{K}}$ a symmetric monoidal structure under $-\otimes -$, the Lurie tensor product  \cite[Corollary 4.8.1.4]{HA}.  

Then, for $\cC, \cD \in \Cat_{\mathcal{K}}$, the $\infty$-category $\cC \otimes \cD$ is generated under $\mathcal{K}$-shaped colimits by the image of the canonical functor $\cC \times \cD \to \cC \otimes \cD$.  
\end{lem}
\begin{proof}
Let $\cE \subseteq \cC \otimes \cD$ be the full subcategory generated under $\mathcal{K}$-colimits by the image of $\cC\times \cD$.  Then by construction, $\cE$ admits a functor from $\cC\times \cD$ which preserves $\mathcal{K}$-colimits separately in each variable, which extends by universal property to a $\mathcal{K}$-colimit preserving functor 
\[
u: \cC \otimes \cD \to \cE.
\]
But the composite of $u$ with the inclusion $\cE\subseteq \cC \otimes \cD$ is the identity of $\cC \otimes \cD$, again by the universal property.  It follows that the inclusion $\cE\subseteq \cC\otimes \cD$ is essentially surjective and therefore an equivalence of $\infty$-categories, as desired.  
\end{proof}

\begin{lem}\label{lem:dbl0}
The subcategory $\calg(\Catsa{m}{p})^{\rig} \subset \calg(\Catsa{m}{p})$ is closed under the symmetric monoidal structure (given by the Lurie tensor product).  
\end{lem}
\begin{proof}
It suffices to show that if $\cC,\cD \in \calg(\Catsa{m}{p})^{\rig}$, then their tensor product $\cC\otimes \cD$ is also rigid.   Note that by construction (\cite[Proposition 4.8.1.10]{HA}), there is a symmetric monoidal functor $\cC \times \cD \to \cC\otimes \cD$, and so any object in the image of $\cC\times \cD$ is dualizable.  But $\cC\otimes \cD$ is generated under $\Sfin{m}{p}$-colimits by the image of $\cC\times \cD$ (\Cref{lem:tensor-gen}), so the lemma follows by noting that for categories in $\calg(\Catsa{m}{p})$, dualizable objects are closed under $\Sfin{m}{p}$-colimits \cite[Corollary 2.7]{Cyclotomic}\footnote{The reference treats the presentable case, and the general case follows from it by using the fully faithful symmetric monoidal $\Sfin{m}{p}$-colimit preserving embedding $\cC \hookrightarrow \mathrm{Fun}^{\Sfin{m}{p}}(\cC^{\op},\Spc )$, where the superscript denotes $\Sfin{m}{p}$-limit preserving functors.}. 
\end{proof}

It follows that $\calg(\Catsa{m}{p})^{\rig}$ inherits a symmetric monoidal structure under the tensor product.  

\begin{lem}\label{lem:dbl1}
The assignment $\cC \mapsto \cC^{\dbl}$ determines a lax symmetric monoidal functor 
\[
\calg(\Catsa{m}{p})\to \calg(\Catsa{m}{p})^{\rig}.
\]
\end{lem}
\begin{proof}
First we note:
\begin{enumerate}
    \item The $\infty$-category $\calg(\Catsa{m}{p})$ can be identified with the (non-full) subcategory of $\calg(\Cat_{\infty})$ spanned by $p$-typical $m$-semiadditive $\infty$-categories in which the tensor product distributes over $\Sfin{m}{p}$-colimits, and symmetric monoidal functors which preserve $\Sfin{m}{p}$-colimits.  
    \item There is a functor $\calg(\Cat_{\infty}) \to \calg(\Cat_{\infty})^{\rig}$ given by extracting dualizable objects.
\end{enumerate}
Thus, to construct the functor $\calg(\Catsa{m}{p})\to \calg(\Catsa{m}{p})^{\rig}$, it suffices to check that the functor of (2) restricts appropriately; in particular we check:
\begin{itemize}
    \item For $\cC \in \calg(\Catsa{m}{p})$, the $\infty$-category $\cC^{\dbl}$ is also $p$-typically $m$-semiadditive.  This is because $\cC^{\dbl} \subset \cC$ is closed under $\Sfin{m}{p}$-colimits \cite[Proposition 2.5]{Cyclotomic}. 
    \item The tensor product in $\cC^{\dbl}$ distributes over $\Sfin{m}{p}$-colimits; this is because the fully faithful embedding $\cC^{\dbl}\subseteq \cC$ is symmetric monoidal and $\Sfin{m}{p}$-colimit preserving, and the tensor product in $\cC$ has this property. 
    
    \item That for $\cC,\cD \in \calg(\Catsa{m}{p})$, a $\Sfin{m}{p}$-colimit preserving symmetric monoidal functor $\cC\to \cD$ restricts to a $\Sfin{m}{p}$-colimit preserving functor on the dualizable objects. This again follows from the fact that $\cC^{\dbl}$ and $\cD^{\dbl}$ are closed under $\Sfin{m}{p}$-colimits in $\cC$ and $\cD$ respectively.
\end{itemize}

It remains to see that this restricted functor is lax symmetric monoidal.  But note that the symmetric monoidal structure on $\calg(\Catsa{m}{p})$ is coCartesian (because it is a category of commutative algebras), and the subcategory $\calg(\Catsa{m}{p})^{\rig}$ is closed under the monoidal structure by \Cref{lem:dbl0}.  Thus, the symmetric monoidal structure on $\calg(\Catsa{m}{p})^{\rig}$ is also coCartesian.  Since the source $\calg(\Catsa{m}{p})$ of the functor is unital, it follows from \cite[Proposition 2.4.3.9]{HA} that the functor is lax symmetric monoidal.  
\end{proof}

We have seen that the construction $\cC \to \cC^{\op}$ defines a symmetric monoidal $C_2$-action on $\Cat_{\infty}$, and that this endows the subcategory $\Catsa{m}{p} \subset \Cat_{\infty}$ with a symmetric monoidal $C_2$-action (cf. \Cref{prop:hc2symmon}).  It follows that $\calg(\Catsa{m}{p})$ also admits a $C_2$-action by $\cC \mapsto \cC^{\op}$.  

\begin{lem}\label{lem:dbl2}
The $C_2$-action on $\calg(\Catsa{m}{p})$ by $\cC \mapsto \cC^{\op}$ restricts to a $C_2$-action on the full subcategory $\calg(\Catsa{m}{p})^{\rig}$, and this restricted action admits a canonical trivialization.  Consequently, there is a lax symmetric monoidal functor 
\[
\calg(\Catsa{m}{p})^{\rig} \to (\calg(\Catsa{m}{p})^{\rig})^{hC_2}.
\]
\end{lem}
\begin{proof}
To restrict the action, it suffices to note that if $\cC$ is rigid, then $\cC^{\op}$ is rigid.  
We have seen (e.g. in the proof of \Cref{lem:dbl1}) that $\calg(\Catsa{m}{p}) \subset \calg(\Cat_{\infty})$ is the inclusion of a subcategory, and it follows from the construction of the $C_2$-action (proof of \Cref{prop:hc2symmon}) that this inclusion is $C_2$-equivariant with respect to the $\op$ action.  Hence, one also has that $\calg(\Catsa{m}{p})^{\rig} \subset \calg(\Cat_{\infty})^{\rig}$ is a $C_2$-equivariant inclusion of a subcategory. 

Now, \cite[Theorem 5.11]{HeineLopezSpitzweck} provides a trivialization of the $C_2$-action on $\calg(\Cat_{\infty})^{\rig}$. Since the subcategory $\calg(\Catsa{m}{p})^{\rig}$ include all the equivalences, such a trivialization necessarily restricts to a trivialization of the $C_2$-action on $\calg(\Catsa{m}{p})^{\rig}$. This trivialization in turn provides a functor 
\[\calg(\Catsa{m}{p})^{\rig} \to (\calg(\Catsa{m}{p})^{\rig})^{BC_2}\simeq (\calg(\Catsa{m}{p})^{\rig})^{hC_2}.
\]
It remains to show that this functor admits a lax symmetric monoidal structure. As in the proof of \Cref{lem:dbl1}, it would suffice to show that the symmetric monoidal structure on $(\calg(\Catsa{m}{p})^{\rig})^{hC_2}$ is coCartesian. Since $\calg(\Catsa{m}{p})^{\rig}$ is given the coCartesian structure, this follows from the general fact that coCartesian symmetric monoidal $\infty$-categories are closed under limits in $\calg(\Cat_\infty)$.
\end{proof}

Now we can finish the proof of \Cref{prop:extractdbl}:

\begin{proof}[Proof of \Cref{prop:extractdbl}]
We construct the desired functor $(-)^{\dbl}$ as a composite
\begin{align*}
\calg(\Catsa{m}{p}) \to \calg(\Catsa{m}{p})^{\rig} &\to \calg(\Catsa{m}{p})^{\rig})^{hC_2} \\
&\to \calg(\Catsa{m}{p})^{hC_2}\to (\Catsa{m}{p})^{hC_2}.
\end{align*}
The first two functors, and their lax symmetric monoidality, are \Cref{lem:dbl1} and \Cref{lem:dbl2}, respectively.  The third functor exists and is symmetric monoidal because, by construction, one has a symmetric monoidal $C_2$-equivariant inclusion $\calg(\Catsa{m}{p})^{\rig} \subset \calg(\Catsa{m}{p})$.  Finally, the last functor is induced by the forgetful functor $\calg(\Catsa{m}{p}) \to \Catsa{m}{p}$, which is $C_2$-equivariantly symmetric monoidal by \Cref{prop:hc2symmon}.  

 

\end{proof}

\subsubsection{Symmetric bilinear forms} \label{subsub:sbf}

We may now refine $\QF(\cC)$ to a higher commutative monoid. 

\begin{defn} \label{defn:QF}
Let $\mdef{\QF^{(-)}}\colon \calg(\Catsa{m}{p}) \to \CMon{m}{p}(\Spc)$ be the composite 
\begin{align*}
\calg(\Catsa{m}{p}) \oto{(-)^\dbl} (\Catsa{m}{p})^{hC_2} &\oto{\InvF} \CMon{m}{p}(\Cat_\infty)^{BC_2}\\
&\oto{(-)^{hC_2}} \CMon{m}{p}(\Cat_\infty) \oto{(-)^\simeq} \CMon{m}{p}(\Spc). 
\end{align*}
\end{defn}
In particular, a point of the space $\QF^{A}(\cC)$ is a pair $(X,b)$ where $X\in \cC^A$ and $b$ is a non-degenerate \emph{symmetric} bilinear form on $X$, with respect to the point-wise tensor product of $\cC^A$. 
We also have the expected integration maps for $\QF^{(-)}(\cC)$:
\begin{prop}
\label{push_quad_correct}
Let $\cC\in \calg(\Catsa{m}{p})$ and let $f\colon A\to B \in \Sfin{m}{p}$. The map 
\[
\push{f}\colon \QF^A(\cC) \to \QF^B(\cC)  
\]
agrees with the map constructed in \Cref{def_push_quadratic_form}. 
\end{prop}

\begin{proof}
Since $\QF^{(-)}(\cC)\simeq (\InvF(\cC^\dbl,\Du)^{\simeq})^{hC_2}$ the claim follows immediately from \Cref{formula_push_anti_inv_fixed}.
\end{proof}

\subsubsection{Higher semiadditive Grothendieck-Witt theory}\label{subsub:gpcomplete}
To produce the higher semiadditive version of Grothendieck-Witt theory from $\QF^{(-)}(\cC)$, all that is left is to perform group completion level-wise. For this, we first need to promote the spaces $\QF^{A}(\cC)$ to commutative monoids in spaces in a way compatible with the higher commutative monoid structure of $\QF^{(-)}(\cC)$. In fact, such an extension is automatic: 

\begin{prop}\label{prop:cmoncmon}
For every $m\ge 0$, there is a canonical equivalence
\[
\CMon{m}{p}(\Spc) \simeq \CMon{m}{p}(\mathrm{CMon}(\Spc))
\]
\end{prop}

\begin{proof}
We have $\CMon{m}{p}(\mathrm{CMon}(\Spc))\simeq \mathrm{CMon}(\Spc)\otimes \CMon{m}{p}(\Spc)$, where $\otimes$ denotes the tensor product of presentable $\infty$-categories \cite[\S 4.8]{HA}. Since $\mathrm{CMon}(\Spc)$ is an idempotent algebra in $\Pr$ classifying $0$-semiadditivity (see, e.g., \cite[Proposition 5.3.1]{AmbiHeight}), the result follows from the fact that for every $m\ge 0$, $\CMon{m}{p}(\Spc)$ is a semiadditive $\infty$-category. 
\end{proof}

Recall that the underlying space functor $\Omega^\infty\colon\Sp\to \Spc$ refines to a functor $\Sp \to \mathrm{CMon}(\Spc)$, whose left adjoint 
$(-)^\gp \colon \mathrm{CMon}(\Spc) \to \Sp$ is the \emph{group completion functor}.

\begin{defn}\label{defn:GW}
For $m\ge 0$, define the functor
\[
\mdef{\GW^{(-)}}\colon \calg(\Catsa{m}{p})\to \PMon{m}{p}(\Sp)
\]
to be the composite
\begin{align*}
\calg(\Catsa{m}{p})\oto{\QF^{(-)}} \CMon{m}{p}(\Spc) &\simeq \CMon{m}{p}(\mathrm{CMon}(\Spc))\\
&\into \PMon{m}{p}(\mathrm{CMon}(\Spc)) \oto{(-)^\gp} \PMon{m}{p}(\Sp).
\end{align*}
\end{defn}

Namely, $\GW^A(\cC)$ is the connective spectrum obtained from
$ \QF(\cC^A) \in \mathrm{CMon}(\Spc)$ via group-completion. Note that the summation operation in the monoid $\QF(\cC^A)$, and hence in the spectrum $\GW^{A}(\cC)$, is given by direct sum of symmetric bilinear forms.

\begin{rem}
Although $\QF^{(-)}(\cC)$ is a $p$-typical $m$-commutative monoid, the functor $\GW^{(-)}(\cC)$ need not satisfy the Segal condition in general, and hence it is only a $p$-typical $m$-commutative \emph{pre-monoid}.  
\end{rem}

\subsubsection{Multiplicative structure on $\GW$} \label{subsub:mult}
Recall that the $\infty$-category $\PMon{m}{p}(\Sp) = \Fun(\Span(\Sfin{m}{p}),\Sp)$ admits a symmetric monoidal structure given by Day convolution. In this section, we show that for every symmetric monoidal $p$-typically $m$-semiadditive $\infty$-category $\cC$, the Grothendieck-Witt object $\GW^{(-)}(\cC)\in \PMon{m}{p}(\Sp)$ admits a canonical structure of a commutative (a.k.a. $\EE_\infty$-) algebra in $\PMon{m}{p}(\Sp)$.  To accomplish this, we refer to \Cref{defn:GW}, in which the functor 
\[\GW^{(-)}: \calg(\Catsa{m}{p}) \to \PMon{m}{p}(\Sp)\] is defined as a certain composite: we will show that each functor in the composite, and consequently $\GW^{(-)}$ itself, is a lax symmetric monoidal functor.  

\begin{prop}\label{prop:QFring}
The functor $$\QF^{(-)} \colon \calg(\Catsa{m}{p}) \to \CMon{m}{p}(\Spc)$$ of \Cref{defn:QF} is lax symmetric monoidal.  Hence, for any symmetric monoidal $p$-typically $m$-semiadditive $\infty$-category $\cC$, the space $\QF(\cC)$ acquires a canonical structure of a commutative algebra in $\CMon{m}{p}(\Spc)$.  
\end{prop}
\begin{proof}
We check that each functor in the composite of \Cref{defn:QF} is lax symmetric monoidal.
\begin{enumerate}
    \item The functor $(-)^{\mathrm{dbl}}$ is lax symmetric monoidal by \Cref{prop:extractdbl}.
    \item The functor $\InvF$ is lax symmetric monoidal by \Cref{fixed_lax}.  
    \item The functor $(-)^{hC_2}$ is lax symmetric monoidal because it is right adjoint to the symmetric monoidal functor $\CMon{m}{p}(\Cat_{\infty}) \to \CMon{m}{p}(\Cat_{\infty})^{BC_2}$ which takes the constant functor on $BC_2$.
    \item The functor $(-)^{\simeq}$ is lax symmetric monoidal because it is right adjoint to the symmetric monoidal inclusion $\CMon{m}{p}(\Spc) \subset \CMon{m}{p}(\Cat_{\infty})$.  
\end{enumerate}
\end{proof}

\begin{cor}\label{cor:GWring}
The functor 
\[
\GW^{(-)}: \calg(\Catsa{m}{p}) \to \PMon{m}{p}(\Sp)
\]
is lax symmetric monoidal.  Hence, for any symmetric monoidal $p$-typically $m$-semiadditive $\infty$-category $\cC$, the spectrum $\GW(\cC)$ acquires a canonical structure of a commutative algebra in $\PMon{m}{p}(\Sp)$.  
\end{cor}
\begin{proof}
We check that the functors in \Cref{defn:GW} are lax symmetric monoidal:
\begin{enumerate}
    \item The functor $\QF^{(-)}$ is lax symmetric monoidal by the previous proposition.  
    \item The equivalence $\CMon{m}{p}(\Spc) \simeq \CMon{m}{p}(\mathrm{CMon}(\Spc))$ is symmetric monoidal by combining (the proof of) \Cref{prop:cmoncmon} with \Cref{cmon_pmon_symm_mon}.
    \item The inclusion $\CMon{m}{p}(\mathrm{CMon} (\Spc)) \into \PMon{m}{p}(\mathrm{CMon} (\Spc))$ is lax symmetric monoidal because it is right adjoint to a symmetric monoidal localization (cf. \Cref{cmon_pmon_symm_mon}).  
    \item The functor $(-)^{\mathrm{gp}}: \mathrm{CMon} (\Spc) \to \Sp$ is symmetric monoidal, and thus the functor $$\PMon{m}{p}(\mathrm{CMon} (\Spc)) \to \PMon{m}{p}(\Sp)$$ induced by post-composition is lax symmetric monoidal for the Day convolution structure \cite[Example 2.2.6.17]{HA}.  
\end{enumerate}
\end{proof}

\subsubsection{The structure on $\GW$ of a discrete ring}

Let $R$ be a discrete commutative ring.  Then the category $\Mod_R(\mathrm{Ab})$ is semiadditive, and therefore by Definitions \ref{defn:QF} and \ref{defn:GW}, we obtain objects
\begin{align*}
    \QF(R) &:= \QF(\Mod_R(\Ab)) \in \CMon{0}{p}(\Spc)\\
    \GW(R) &:= \GW(\Mod_R(\Ab))\in \Sp.
\end{align*}

Moreover, by \Cref{prop:QFring} and \Cref{cor:GWring}, since $\Mod_R(\mathrm{Ab})$ is semiadditively \emph{symmetric monoidal}, the objects $\QF(R)$ and $\GW(R)$ naturally acquire the structure of $\mathbb{E}_{\infty}$-algebras in their respective categories.   In fact, one can explicitly describe the resulting ring structures on  $\pi_0\QF(R)$ and $\GW_0(R)$ acquire natural ring structures: if $(V, q),(V', q')\in \pi_0\QF(R)$ are symmetric bilinear forms, then their sum is given by $(V\oplus V', q\oplus q')$ and their product by $(V\otimes V', q\otimes q')$.  

The payout of our work in \Cref{sec:higher_QG} is the following extension of this situation:


\begin{example}\label{exm:qfRcmon1}
Let $R$ be a discrete commutative ring in which $2$ is invertible.  Then the category $\Mod_R(\mathrm{Ab})$ is $2$-typically $1$-semiadditive\footnote{This follows, for instance, by \cite[Proposition 3.2.2]{AmbiHeight}, noting that this category has height $0$ because $2$ is invertible.}.  Therefore, by Definitions \ref{defn:QF} and \ref{defn:GW}, $\QF(R)$ and $\GW(R)$ extend to functors
\begin{align*}
    \QF^{(-)}(R) &: \Span(\Sfin{1}{2})^{\op} \to \Spc\\
    \GW^{(-)}(R) &: \Span(\Sfin{1}{2})^{\op} \to \Sp
\end{align*}
where $\QF^{(-)}(R)$ additionally satisfies the Segal condition.  Moreover, since $\Mod_R(\Ab)$ is in fact semiadditively symmetric monoidal, we have $\QF^{(-)}(R) \in \calg(\CMon{1}{2}(\Spc))$ and $\GW^{(-)}(R) \in \calg(\PMon{1}{2}(\Sp))$ by \Cref{prop:QFring} and \Cref{cor:GWring}.  
\end{example}

\section{$K(1)$-local Grothendieck-Witt theory of finite fields}\label{sect:gwfield}
Let $\Sph_{K(1)}$ denote the $K(1)$-local sphere at the prime $2$.  Recall that, in order to prove \Cref{thm:main}, we aim to compute the element $|BC_2| \in \pi_0(\Sph_{K(1)})$, which is defined in terms of the $1$-commutative monoid structure on $\Sph_{K(1)}$, that is, a certain functor
\[
\Sph_{K(1)}^{(-)}: \Span(\Spc_1^{(2)})^\op \to \Sp.
\]  
The goal of this section is to show that this functor can be understood in terms of the functor $\GW^{(-)}$ constructed in \Cref{sec:higher_QG}.  In particular, we will show: 

\begin{thm}\label{sk1main}
Let $\ell \equiv 3,5\pmod{8}$ be a prime, and let 
\[
\GW^{(-)}(\FF_\ell) :  \Span(\Spc_1^{(2)})^\op \to \Sp
\]
denote the functor constructed in \Cref{exm:qfRcmon1}.  Then there is a natural transformation of functors $\GW^{(-)}(\FF_{\ell})^{\wedge}_2 \to \Sph_{K(1)}^{(-)}$ which, pointwise, exhibits the source as the connective cover of the target.  
\end{thm}

The proof of \Cref{sk1main}, which takes up the bulk of this section, will be outlined and given in \Cref{sub:sk1mainpf}.  

\begin{rem}
A special case of our theorem, given by evaluating at a point, is an identification 
\begin{equation}\label{eqn:gwk1sph}
    L_{K(1)} \GW (\FF_{\ell}) \simeq \Sph_{K(1)}
\end{equation} for primes $\ell \equiv 3,5\pmod{8}$.  This equivalence is essentially due to Friedlander \cite{Friedlander}.  As we will see, the proof of our theorem draws heavily on later work of Fiedorowicz-Hauschild-May \cite{FHM}, which generalizes (\ref{eqn:gwk1sph}) to the equivariant setting. 
\end{rem}


\subsection{Relating $\GW(\FF_\ell)$ to $\Sph_{K(1)}$}\label{sub:sk1mainpf}

Our proof of \Cref{sk1main} proceeds in roughly two steps.  First, we start by working with the Grothendieck-Witt theory of $\cl{\FF}_\ell$, rather than $\FF_\ell$; in \Cref{subsub:ko}, we use work of Fiedorowicz-Hauschild-May \cite{FHM} to show that $L_{K(1)}\GW^{(-)}(\cl{\FF}_\ell)$ is a $1$-commutative monoid (in fact, we will see that it is the unique $1$-commutative monoid structure on $KO^{\wedge}_2$).  Then, in \Cref{subsub:frob}, we pass to Frobenius fixed points and finish the proof of the theorem.

\subsubsection{Relating $\GW(\cl{\FF}_{\ell})$ to $KO$}\label{subsub:ko}

Our goal in \Cref{subsub:ko} will be to show:

\begin{prop}\label{prop:kosegal}
For any prime $\ell$, the functor \[L_{K(1)}\GW^{(-)}(\cl{\FF}_\ell) \colon \Span(\Spc_1^{(2)})^\op \to \Sp \] satisfies the Segal condition (cf. \Cref{defn:cmon}). In other words, 
$L_{K(1)}\GW^{(-)}(\cl{\FF}_\ell)$ defines a $2$-typical $1$-commutative monoid in $K(1)$-local spectra.\end{prop}

This proposition follows essentially immediately from results of Fiedorowicz-Hauschild-May \cite{FHM} and the Atiyah-Segal completion theorem, which we now state.

\begin{prop}[Fiedorowicz-Hauschild-May {\cite[Theorem 0.3]{FHM}}]\label{prop:rescompat}
Let $\ell$ be an odd prime, let $G$ be a finite $2$-group, and let $KO^G$ denote the (genuine) $G$-fixed points of the equivariant real $K$-theory spectrum. Then there is an equivalence\footnote{Note that while $2$-completion and connective cover do not commute in general, they do when the spectra in question are of finite type.  Since we will apply these functors in the finite type setting, we use a notation which does not distinguish the order of applying them.}
\[
\beta^G: \GW^{G}(\cl{\FF}_\ell)^{\wedge}_2 \to \tau_{\geq 0}(KO^G)^{\wedge}_2.
\]
Moreover, these equivalences are compatible with restriction in $G$ in the sense that there is a commutative square
\[\begin{tikzcd}
\GW^G(\cl{\FF}_\ell)^{\wedge}_2 \arrow[r,"\rho_{BG}"] \arrow[d,"\beta^G"] &  (\GW(\cl{\FF}_\ell)^{\wedge}_2)^{hG} \arrow[d,"\beta^{hG}"]\\
\tau_{\geq 0}(KO^G)^{\wedge}_2 \arrow[r] & (\tau_{\geq 0} KO^{\wedge}_2)^{hG}\\
\end{tikzcd} \]
where $\rho_{BG}$ denotes the Segal map (\Cref{defn:cmon}) and the bottom arrow is the canonical map induced by the genuine $G$-equivariant structure.  
\end{prop}

\begin{prop}[Atiyah-Segal Completion Theorem, $2$-complete form \cite{AtiyahSegal}]\label{atiyah-segal}
Let $G$ be a $2$-group.  Then the natural map of spectra $KO^G \to KO^{BG}$ is an equivalence after $2$-completion.  
\end{prop}
\begin{proof}
Since the spectra in question are of finite type, the homotopy groups of the $2$-adic completions are the $2$-completion of the homotopy groups, it suffices to show this is a $2$-complete equivalence on homotopy groups.  By the usual Atiyah-Segal completion theorem \cite{AtiyahSegal}, this map exhibits $\pi_*(KO^{BG})$ as the completion of the ring $\pi_*(KO^G)$ at the augmentation ideal.  Since $G$ is a $2$-group, the augmentation ideal defines the $2$-adic topology (see \cite[Proposition III.1.1]{AtiyahTall}).  \end{proof}

Given these, the proof of \Cref{prop:kosegal} is just unwinding definitions.

\begin{proof}[Proof of \Cref{prop:kosegal}]
Since the functor takes disjoint unions of spaces to products of spectra, it suffices to check the Segal condition on connected spaces in $\Sfin{1}{2}$, i.e., spaces of the form $BG$ for a finite $2$-group $G$.  In this case, we would like to show that the map
\[
\rho_{BG}: L_{K(1)}\GW^{G}(\cl{\FF}_\ell) \to L_{K(1)}\GW(\cl{\FF}_{\ell})^{BG}
\]
induced by the point embeddings $\pt \to BG$ is an equivalence.  By \Cref{prop:rescompat}, this amounts to showing that the natural map
\[
 \tau_{\geq 0}(KO^G)^{\wedge}_2 \to (\tau_{\geq 0}KO^{\wedge}_2)^{BG}
\]
is an equivalence after $K(1)$-localization for every $2$-group $G$.  But $(KO^G)^{\wedge}_2$ is $K(1)$-local and $K(1)$-localization is insensitive to connective cover, so the left-hand side localizes to $(KO^G)^{\wedge}_2$.  For the right-hand side, we have
\[
L_{K(1)} (\tau_{\geq 0}KO^{\wedge}_2)^{BG} \simeq  L_{K(1)}(KO^{\wedge}_2)^{BG} \simeq  (KO^{\wedge}_2)^{BG}\simeq (KO^{BG})^{\wedge}_2 ,
\]
where the first equivalence is because $K(1)$-localization only depends on the connective cover, the second is because $K(1)$-local spectra are closed under limits in $\Sp$, and the third is because the formula for $2$-completion $X^{\wedge}_2 = \lim_i X/2^i$ implies that it commutes with limits.  Thus, the conclusion follows from \Cref{atiyah-segal}.  
\end{proof}

\begin{rem}\label{rem:ko1comm}
Note that, by \Cref{prop:rescompat}, the $2$-typical $1$-commutative monoid $\GW^{(-)}(\cl{\FF}_\ell)$ of \Cref{prop:kosegal} has underlying object $KO^{\wedge}_2$.  Although we will not need this explicitly, we remark that, by the uniqueness of $1$-commutative monoid structures on $K(1)$-local spectra (\Cref{prop:highercommsemiadd}(3)), this means that  $\GW^{(-)}(\cl{\FF}_\ell)$ and  $(KO^{\wedge}_2)^{(-)}$ (cf. \Cref{defn:canhighercm}) are equivalent as $2$-typical $1$-commutative monoids.  
\end{rem}

\subsubsection{Fixed points of Frobenius}\label{subsub:frob}


Note that the Frobenius automorphism induces an action of the group $\mathbb{Z}$ on $\cl{\FF}_\ell$, and thus on the functor $\GW^{(-)}(\cl{\FF}_\ell)$.  We denote the corresponding natural transformation by
\[
\varphi \colon \GW^{(-)}(\cl{\FF}_\ell) \to \GW^{(-)}(\cl{\FF}_\ell).
\]
 On the other hand, the inclusion $\FF_\ell \to \cl{\FF}_{\ell}$ is $\mathbb{Z}$-equivariant (with respect to the trivial action on the source), and thus we have an induced natural transformation
 \[
 u: \GW^{(-)}(\FF_\ell) \to \GW^{(-)}(\cl{\FF}_\ell)^{h\ZZ} \simeq \mathrm{fib}(1-\varphi).
 \]
The work of Fiedorowicz-Hauschild-May \cite{FHM} shows that $u$ is close to an equivalence:

\begin{prop}\label{prop:frobcompat}
Let $\beta^G: \GW^{G}(\cl{\FF}_\ell)^{\wedge}_2 \to \tau_{\geq 0}(KO^G)^{\wedge}_2$ denote the equivalence of \Cref{prop:rescompat}.  Then we have:
\begin{enumerate}
    \item Under the equivalence $\beta^G$, the Frobenius automorphism $\varphi \colon \GW^{G}(\cl{\FF}_\ell) \to \GW^{G}(\cl{\FF}_\ell)$ is identified with the map $\psi^\ell: \tau_{\geq 0}(KO^G)^{\wedge}_2\to\tau_{\geq 0}(KO^G)^{\wedge}_2$ induced by the $\ell$-th Adams operation.  
    \item The $2$-completion of the map $u$ exhibits the functor $\GW^{(-)}(\FF_\ell)^{\wedge}_2$ as  the (pointwise) connective cover of $(\GW^{(-)}(\cl{\FF}_\ell)^{\wedge}_2)^{h\ZZ}$.
\end{enumerate}
\end{prop}
\begin{proof}
Part (1), the identification of the Frobenius with the Adams operation, is demonstrated in \cite[Section 8]{FHM}, in the proof of \cite[Theorem 0.5]{FHM}.    Part (2) is \cite[Theorem 8.1]{FHM} (note that the connective cover comes from the fact that their theorem is stated in spaces).
\end{proof}

We have the following immediate consequence:

\begin{lem}\label{lem:conncov}
The $K(1)$-localization map $\GW^{(-)}(\FF_\ell)^{\wedge}_2 \to L_{K(1)} \GW^{(-)}(\FF_\ell)$ pointwise exhibits the source as the connective cover of the target.
\end{lem}
\begin{proof}
We have
\begin{align*}
\GW^{(-)}(\FF_\ell)^{\wedge}_2 &\simeq \tau_{\geq 0}(\GW^{(-)}(\cl{\FF}_\ell)^{\wedge}_2)^{h\ZZ}\\
&\simeq \tau_{\geq 0}(\tau_{\geq 0}(KO^G)^{\wedge}_2)^{h\ZZ}\\
&\simeq \tau_{\geq 0}((KO^G)^{\wedge}_2)^{h\ZZ}\\
&\simeq  \tau_{\geq 0}L_{K(1)}\tau_{\geq 0}(\tau_{\geq 0} (KO^G)^{\wedge}_2)^{h\ZZ}\\
&\simeq \tau_{\geq 0} L_{K(1)}\GW^{G}(\FF_\ell),
\end{align*}
where the first two equivalences use \Cref{prop:frobcompat}, the third is by connectivity considerations, the fourth is because $L_{K(1)}$ is zero on coconnective spectra, and the last equivalence is by the first two equivalences.  
\end{proof}

We may now finish the proof of \Cref{sk1main}, with the critical input being  \Cref{prop:highercommsemiadd}(3), which asserts that $K(1)$-local spectra admit an essentially unique $1$-commutative monoid structure.  

\begin{proof}
By \Cref{prop:frobcompat}(2), the map of functors
\[
u:  \GW^{(-)}(\FF_\ell) \to \GW^{(-)}(\cl{\FF}_\ell)^{h\ZZ}
\]
is an equivalence after applying $L_{K(1)}$.  Thus, since $L_{K(1)} \GW^{(-)}(\cl{\FF}_\ell)$ satisfies the Segal condition, so does $ L_{K(1)}\GW^{(-)}(\FF_\ell)$ (note that $(-)^{h\ZZ}$ is a finite limit so it can be applied before or after $K(1)$-localization with the same effect).  Finally, applying \Cref{prop:frobcompat}(1), we see that the underlying space of the resulting $2$-typical $1$-commutative monoid  $ L_{K(1)}\GW^{(-)}(\FF_\ell)$ is the fiber of the map
\[
1-\psi^l : KO^{\wedge}_2 \to KO^{\wedge}_2.  
\]
When $\ell\equiv 3,5\pmod{8},$ $\ell$ is a topological generator for $\ZZ_2^{\times}/\{\pm 1\}\simeq \ZZ_2$ and therefore this fiber is equivalent to $\Sph_{K(1)}$ (see, for instance, \cite{HopkinsK1}).  By \Cref{prop:highercommsemiadd}(3), we conclude that there is an equivalence of $2$-typical $1$-commutative monoids
\[
L_{K(1)}\GW^{(-)}(\FF_{\ell}) \simeq \Sph_{K(1)}^{(-)}.
\]
The theorem then follows by \Cref{lem:conncov}.  
\end{proof}

\section{Computations in the $K(1)$-local sphere}\label{sect:computations}
Let $p=2$ and let $\ell$ be a prime congruent to $3$ or $5$ modulo $8$. 
The fact that the $2$-typical $1$-commutative monoid $\Sph_{K(1)}^{(-)}$ is the $K(1)$-localization of the $1$-commutative pre-monoid $\GW^{(-)}(\FF_\ell)$ allows us to deduce facts about the $K(1)$-local sphere from more concrete computations in symmetric bilinear forms over the finite field $\FF_\ell$.  

We start by briefly recalling some aspects of the classical theory of $\GW_0(\FF_{\ell})$ and explicitly relating $\GW_0(\FF_{\ell})$ to $\pi_0\Sph_{K(1)}$ in \Cref{sub:pi0map}.  Our first application of this relationship is to compute the cardinalities of $2$-typical $\pi$-finite spaces in $\pi_0\Sph_{K(1)}$ (\Cref{thm:main}) in \Cref{sub:cardinalities}.

We then turn to computing various natural power operations on $\pi_0\Sph_{K(1)}$.  Because of the multiplicative nature of the $K(1)$-local equivalence $L_{K(1)}\GW^{(-)}(\FF_\ell) \simeq \Sph_{K(1)}^{(-)}$, these power operations can be expressed in terms of multiplicative operations on symmetric bilinear forms over $\FF_\ell$.  The first of these operations is the operation $\alpha_p$ (studied in \cite{TeleAmbi}), which we compute in \Cref{sub:alpha}.  In \Cref{sub:thetadelta}, we use $\alpha_p$ to compute two closely related operations: the operation $\theta$ (originally due to McClure, cf. \cite{Hinfty,HopkinsK1}), and the canonical $p$-derivation $\delta_p$, which was used in the proof of the higher semiadditivity of $\Sp_{T(n)}$ in \cite{TeleAmbi}.  Finally, we compute Rezk's logarithm on $\pi_0\Sph_{K(1)}^{\times}$ in \Cref{subsub:log}.

\subsection{Relating $\GW_0(\FF_\ell)$ to $\pi_0\Sph_{K(1)}$}\label{sub:pi0map}
The purpose of this section will be to compile some of the explicit consequences of our work to this point.  We will review the theory of symmetric bilinear forms over a finite field in \Cref{subsub:qf-fin-field}.  Then, in \Cref{subsub:gwfield}, we specialize the results of \Cref{sec:higher_QG} to the case of $\GW$ of a discrete ring $R$ in which $2$ is invertible and describe the resulting transfer maps coming from the $1$-commutative monoid structure on $\QF(R)$.  Finally, in \Cref{subsub:map-pi0}, we explicitly understand the map $\GW_0(\FF_\ell) \to \pi_0 \Sph_{K(1)}$ induced by the identification of \Cref{sk1main}, which will allow us to deduce how the higher semiadditive transfers act on $\pi_0\Sph_{K(1)}$.  

\subsubsection{Symmetric bilinear forms over a discrete ring}\label{subsub:qf-fin-field}


\begin{notation}\label{ntn:canform}
Let $R$ be a commutative ring.  Then every invertible element $r\in R^\times$ determines a nondegenerate symmetric bilinear form 
\[
b_r(x\otimes y) = rxy,
\]
and we denote the class of $(R,b_r)$ in $\pi_0\QF(R)$ or $\GW_0(R)$ by $[r]$.  Note that the construction $r\mapsto [r]$ determines a group homomorphism $R^{\times} \to \GW_0(R)$.  
\end{notation}

The monoid $\pi_0\QF(R)$ can be explicitly described when $R$ is a finite field in which $2\neq 0$.  

\begin{example}\label{exm:finitefield}
Let $k$ be a finite field of characteristic not equal to $2$.  Then any nondegenerate symmetric bilinear form $(V, q)$ can be diagonalized, and therefore splits as a sum
\[
(V,q) = \sum_i [x_i]
\]
for elements $x_i \in k^{\times}/(k^{\times})^2$.  Since $k$ is a finite field, this latter group is isomorphic to $\mathbb{Z}/2$, generated by a nonsquare class $r$.  Moreover, the only relation is $2[r] = 2$, so the monoid $\QF_0(k)$  of symmetric bilinear forms over $k$ is the free commutative monoid on $[1]$ and $[r]$ subject to the relation $2[r] = 2[1]$.  Moreover, the multiplicative structure satisfies $[r]^2 = [1]$.

This monoid comes with two homomorphisms
\begin{align*}
    \mathrm{rank}: \QF_0(k) &\to \mathbb{N} && \mathrm{det}:\QF_0(k)\to k^{\times}/(k^{\times})^2\\
    a[1]+b[r] &\mapsto a+b  &&  \quad   \ a[1]+b[r] \mapsto r^b
\end{align*}
whose product is an injective map $\QF_0(k) \xrightarrow{(\mathrm{rank},\mathrm{det})} \mathbb{N} \times \mathbb{Z}/2$.  That is, any symmetric bilinear form can be recovered from its rank and determinant.  
\end{example}

From this example, one can read off the $0$-th Grothendieck-Witt group (see, e.g., \cite[Theorem 3.5]{lam2005introduction}):

\begin{prop}
\label{GW_finite_field}
Let $k$ be a finite field of odd characteristic.  Then
\[
\pi_0\GW(k)\simeq \ZZ[e]/(e^2,2e),
\] 
where $e = [r] - [1]$ for $r\in k^\times$ a non-square. 
\end{prop}

\subsubsection{The higher semiadditive integration maps on $\GW_0$}\label{subsub:gwfield}

Let $R$ be a commutative ring in which $2$ is invertible.  Then we have seen in \Cref{exm:qfRcmon1} that the $2$-typical $1$-semiadditive structure on $\Mod_R(\Ab)$ gives us functors
\begin{align*}
    \QF^{(-)}(R) &: \Span(\Sfin{1}{2})^{\op} \to \Spc\\
    \GW^{(-)}(R) &: \Span(\Sfin{1}{2})^{\op} \to \Sp.
\end{align*}
The following proposition describes explicitly the integration maps for $\QF^{(-)}(R)$ (and thus for $\GW^{(-)}(R)$). 



\begin{prop}
\label{formula_push_quad}
Let $G$ be a finite $2$-group, let $R$ be a commutative ring in which $2$ is invertible, and let $(V,b)\in \QF^{BG}(R)$ be a $G$-equivariant $R$-module equipped with an equivariant symmetric bilinear form.  

Then the map
\[
\push{BG} \colon \QF^{BG}(R) \to \QF(R)
\] 
sends the pair $(V,b)$ to the pair $(V_G, \push{BG}b)$ where the symmetric bilinear form $\push{BG}b$ is given by the formula

\[
\push{BG} b(\overline{u},\overline{v}) = 
\sum_{g\in G}b(gv, u).
\]
Here, $V_G$ denotes the $G$-coinvariance of $V$ and $\overline{u},\overline{v}\in V_{G}$ are the images of $u,v\in V$ under the quotient map.
\end{prop}

\begin{proof}
The underlying module is $V_G$ because $\push{BG}$ is, by construction, given by left Kan extension along $BG\to \pt$ on the underlying object.  To determine the bilinear form, let $b^\vee\colon V\to \Du{V}$ be the mate of $b$. Then, by \Cref{push_quad_correct}, the map $(\smallint_{BG}b)^\vee$ is given by
\[(\smallint_{BG}b)^\vee(\overline{u}) = b^\vee(\Nm_{BG}(\overline{u})) = b^\vee(\sum_{g\in G} gu)\]
so that
\[
(\smallint_{BG}b)(\overline{u},\overline{v}) = 
((\smallint_{BG}b)^\vee(\overline{u}))(\overline{v}) = \big( b^\vee(\sum_{g\in G} gu) \big) (v) = \sum_{g\in G}b(gu,v). 
\]
\end{proof}


\begin{cor}
\label{int_const_quad}
For any finite $2$-group $G$, we have
\[
\push{BG}[r] = [|G|\cdot r] \qin \QF(R).
\]
\end{cor}

\begin{proof}
By \Cref{formula_push_quad} and because $G$ acts trivially on $R$, we have 
\[
(\push{BG}b_r)(x,y) = \sum_{g\in G} b_r(gx,y) = \sum_{g\in G}b_r(x,y) =|G|\cdot b_r(x,y) = b_{|G| r}(x,y), 
\]
which is exactly the symmetric bilinear form on $R$ corresponding to the object $[|G|\cdot r]\in \QF(R)$.
\end{proof}

\subsubsection{Relating $\GW_0(\FF_\ell)$ and $\pi_0\Sph_{K(1)}$}\label{subsub:map-pi0}

\Cref{sk1main} roughly asserts that the $1$-semiadditive structure of $\Sph_{K(1)}$ can be understood in terms of the higher semiadditive Grothendieck-Witt theory of a finite field via a certain natural transformation $\GW^{(-)}(\FF_\ell) \to \Sph_{K(1)}^{(-)}$.  In order to better translate between the two objects, we will compute $\GW(\FF_\ell) \to \Sph_{K(1)}$ explicitly on $\pi_0$.

First, recall from \Cref{subsubsec:sk1} the equivalence
\[
\pi_0 \Sph_{K(1)}\simeq \ZZ_2[\varepsilon] / (\varepsilon^2, 2\varepsilon),
\]
where $\varepsilon :=\eta \cdot \zeta$.  On the other hand, by \Cref{GW_finite_field}, one has an isomorphism
\[
\pi_0\GW(\FF_\ell) \simeq \ZZ[e]/(e^2,2e)
\]
where $e$ represents the class $[r] - [1]$ for a nonsquare $r\in \FF_\ell^{\times}.$  

\begin{lem}\label{lem:pizero}
Let $\ell\equiv 3,5\pmod{8}$ be prime.  Then the ring map
\[
\ZZ[e]/(e^2,2e) \simeq \pi_0\GW(\FF_\ell) \to \pi_0\Sph_{K(1)} \simeq \ZZ_2[\varepsilon]/(\varepsilon^2,2\varepsilon)
\]
of \Cref{sk1main} sends $e$ to $\varepsilon$.  
\end{lem}
\begin{proof}
By \Cref{sk1main}, the map exhibits the target as the $2$-completion of the source so $e$ cannot go to zero. Since $e$ is the only nilpotent element in $\ZZ[e]/(e^2,2e)$ and $\varepsilon$ is the only nilpotent element of $\ZZ_2[\varepsilon]/(\varepsilon^2,2\varepsilon)$, we deduce that this homomorphism must take $e$ to $\varepsilon$.
\end{proof}
 
We then have the following immediate corollary of \Cref{lem:pizero} and \Cref{exm:finitefield}.  

\begin{cor}
Let $l \equiv 3,5\pmod{8}$ be prime and suppose that $(V,q) \in \QF(\ell)$ such that 
\[
\mathrm{rank}(V,q) = r, \quad \mathrm{det}(V,q)=d.
\]
Then the image of $(V,q)$ under the composite
\[
\QF_0(\FF_\ell) \to \GW_0(\FF_{\ell}) \to \pi_0\Sph_{K(1)}
\]
is given by $r+ d\varepsilon.$
\end{cor}


\subsection{$K(1)$-local cardinalities of $2$-typical $\pi$-finite spaces}\label{sub:cardinalities}
Recall that, for a $\pi$-finite space $A$, one can associate an element $|A|\in \pi_0 \Sph_{K(1)}$, known as the ($K(1)$-local) cardinality of $A$ (cf. \Cref{defn:card}). In this section, we explain how to compute these cardinalities for every $2$-typical $\pi$-finite space $A$.   In fact, it suffices to do this in the case $A$ is connected, as if $A = \sqcup_i A_i$, then $|A|=\sum_{i} |A_i|$.  For connected $A$, we shall express the result in terms of its \mdef{\textit{homotopy cardinality}}, defined by the formula
\[
    \mdef{|A|_0} := \prod_{i\in \NN} |\pi_i A|^{(-1)^i}.
\]
We have: 




\begin{thm} \label{card_BC_2}
    Let $\Sph_{K(1)}$ denote the $K(1)$-local sphere at the prime 2, and let $A$ be a connected $2$-typical $\pi$-finite space. Then we have
\begin{equation} \label{eq:card_formula_2_typical}
    |A| = 1+\log_2(|A|_0)\cdot \varepsilon \qin \pi_0\Sph_{K(1)}.
\end{equation}
In particular,
\[
    |BC_2| = 1 + \varepsilon.
\]
\end{thm}
\begin{proof}
    We first reduce the general case to the case $A=BC_2$. 
    Let $A\to B \to C$ be a principal fiber sequence of $2$-typical $\pi$-finite spaces such that $A$ is connected. Then, since $\Sp_{K(1)}$ is of semiadditive height 1, we have by \cite[Theorem A(4)]{AmbiHeight} that 
\[
    |B| = |A|\cdot|C| \qin \pi_0\Sph_{K(1)}.
\]
    On the other hand, since we have $|B|_0 = |A|_0|C|_0$ (either by the long exact sequence on homotopy or by the same reference), and since $\varepsilon^2 = 0$, we deduce that 
\[
    1 + \log_2(|B|_0)\varepsilon = 1 + \log_2(|A|_0)\varepsilon + \log_2(|C|_0)\varepsilon = (1 + \log_2(|A|_0)\varepsilon) \cdot (1 + \log_2(|C|_0)\varepsilon).      
\]
    Hence, both the left hand side and the right hand side in the proposed identity  (\ref{eq:card_formula_2_typical}) are multiplicative in principal fiber sequences with connected fiber.  
    
    Note that every connected $2$-typical $\pi$-finite space $A$ is nilpotent, and hence participates in a sequence of principal fibrations 
\[
    A = A_0 \to A_1 \to ... \to A_k = \pt 
\]
    such that, for every $0 \le i \le k-1$, the fiber of the map $A_i \to A_{i+1}$ is of the form $B^{\ell_i} C_2$ for some $\ell_i > 0$.  Hence, the claim for general $A$ follows from the cases $A=B^\ell C_2$. Moreover, since we have principal fiber sequences 
    \[
    B^{\ell-1}C_2 \to \pt \to B^\ell C_2,
    \]
we may further reduce these cases to the single case $A=BC_2$. 
 
We now prove the claim for $BC_2$.  
By the results of \Cref{sect:gwfield}, we have a (unital) map 
\[
\GW^{(-)}(\FF_\ell) \to \Sph_{K(1)}^{(-)}\]
in $\PMon{1}{2}(\Sp)$.  Thus, by definition of the cardinality, the map 
$\pi_0\GW(\FF_\ell) \to \pi_0\Sph_{K(1)}$ carries the element $|BC_2|\in \pi_0\GW(\FF_\ell)$ to the element $|BC_2|\in \pi_0\Sph_{K(1)}$. Therefore, it is enough to compute $|BC_2|$ in $\pi_0\GW(\FF_\ell)$.  Moreover, since (by \Cref{lem:pizero}) the map $\pi_0\GW(\FF_\ell) \to \pi_0\Sph_{K(1)}$ sends $e$ to $\varepsilon$, it suffices to see that 
\[
|BC_2| = 1 + e \qin \pi_0\GW(\FF_\ell).  
\]

Let $\pi: BC_2 \to \pt$ denote the terminal map.  By definition, $|BC_2| \in \pi_0 \GW(\FF_\ell)$ is given by $\push{\pi}\pi^* 1$, where $1 \in \pi_0 \GW(\FF_\ell)$ is represented by the symmetric bilinear form $[1] \in \QF(\FF_\ell)$.  Under $\pi^*: \GW(\FF_\ell) \to \GW^{BC_2}(\FF_\ell)$, the element $1\in \pi_0 \GW(\FF_\ell)$ is sent to the element  $1\in \pi_0\GW^{BC_2}(\FF_\ell)$ represented by the symmetric bilinear form $[1]\in \QF(\FF_\ell)$, thought of as a $C_2$-equivariant form via the trivial $C_2$-action on $\FF_\ell$. 
Thus, by \Cref{int_const_quad} we have 
\[
|BC_2| = \int\limits_{BC_2} [1] = [|C_2|\cdot 1] = [2] = [1] + ([2] - [1])\qin \pi_0\GW(\FF_\ell).
\]
Finally, since $\ell\equiv 3,5 \pmod{8}$ by assumption, $2$ is not a square in $\FF_\ell^{\times}$, and so $[2]-[1]$ is a representative for $e$ in $\GW_0(\FF_\ell)$ and the result follows.
\end{proof}



\subsection{The operation $\alpha$} \label{sub:alpha}
Let $(\cC, \one_{\cC})$ be a symmetric monoidal $\infty$-category, let $R\in \calg(\cC)$, and let $\Sigma_p$ denote the symmetric group on $p$ letters. Then, consider the natural map of sets
\[
\pi_0(R) := \pi_0 \Map(\one_{\cC}, R) \xrightarrow{(-)^{\otimes p}} \pi_0 \Map(\one_{\cC} , R^{\otimes p})^{h\Sigma_p} \to \pi_0 \Map( \one_{\cC}, R)^{h\Sigma_p} \cong \pi_0 R^{B\Sigma_p},
\]
where the first arrow is induced by taking $p$th tensor power and the second by multiplication.  We refer to this map as the \emph{total $p$-power operation} and denote it by 
\[
\PP_p \colon \pi_0(R) \to \pi_0(R^{B\Sigma_p}).
\]
This map can be thought of as taking $x\in \pi_0(R)$ to the canonical $B\Sigma_p$-family of $p$-th powers of $x$. If $\cC$ is $p$-typically $1$-semiadditive, then one can integrate $\PP_p$ over the classifying space $BC_p$ of the $p$-Sylow subgroup $C_p \subseteq \Sigma_p$, to obtain a self map of $\pi_0(R)$.
Let $i\colon BC_p \into B\Sigma_p$ be the map induced by the inclusion of a $p$-Sylow subgroup of $\Sigma_p$.
\begin{defn}[\cite{TeleAmbi} \S 4.2] \label{defn:alpha}
Let $\cC \in \calg(\Catsa{1}{p})$. For $R\in \calg(\cC)$, we define the map 
$\alpha_p\colon \pi_0(R)\to \pi_0(R)$ by the composite
\[
    \mdef{\alpha_p} \colon \pi_0(R) \oto{\PP_p} \pi_0(R^{\Sigma_p}) \oto{i^*} \pi_0(R^{BC_p})  \oto{\push{BC_p}} \pi_0(R),
\]
\end{defn}
Informally, we have 
\[
\alpha_p(x) = \push{BC_p}x^p.
\]
\begin{rem} \label{rem:functional_equation_alpha}
Morally, $\alpha_p$ is supposed to resemble the function $\frac{x^p}{p}$. For example, one can show \cite[\S 4.2]{TeleAmbi} that $\alpha_p$ satisfies the functional equation
\begin{equation}\label{eqn: func_eq_alpha}
\alpha_p(x+y) - \alpha_p(x)-\alpha_p(y) = \frac{(x+y)^p - x^p-y^p}{p}.
\end{equation}
Here, the right hand side does not actually involve division by $p$, and hence makes sense in every ring. 
\end{rem}

\begin{rem}\label{rem:alpha_cardinality}
The operation $\alpha_p$ has a simple interaction with cardinalities, namely that
\[
\alpha_p(|A|) = |A\wr C_p| = |(A^{\times C_p})_{hC_p}|,
\]
where $C_p$ acts on $A^{\times C_p}$ by permuting the factors \cite[Theorem 4.2.12]{TeleAmbi}. 
\end{rem}

When combined with the functional equation from \Cref{rem:functional_equation_alpha}, this fact gives a computation of $\alpha_p$ on $\ZZ_p = \pi_0\Sph_{K(1)}$.

\begin{thm} \label{alpha_trick}
Let $\Sph_{K(1)}$ denote the $K(1)$-local sphere at the prime $p$, and for $p=2$ let $\varepsilon = \eta \cdot \zeta \in \pi_0\Sph_{K(1)}$. 
For an odd prime $p$, we have 
\[
\alpha_p(x) = \frac{x^p +(p-1)x}{p} \qin \pi_0\Sph_{K(1)}.
\]
For $p=2$, we have
\[
\alpha_2(r + d \varepsilon) = \frac{r^2 + r}{2} + (rd +r + d)\varepsilon \qin \pi_0\Sph_{K(1)}.
\]
\end{thm}

\begin{proof}
First, a direct computation shows that the proposed formulas for $\alpha_p$ satisfy the functional equation in \Cref{rem:functional_equation_alpha}.  Note that every two functions $\pi_0 \Sph_{K(1)} \to \pi_0 \Sph_{K(1)}$ which satisfy this equation differ by a function $\pi_0 \Sph_{K(1)} \to \pi_0 \Sph_{K(1)}$ which is additive.   Hence, we may a priori write $\alpha_p$ as the sum of the proposed formula and an additive term.  

Additionally, note that any additive function $\ZZ_p \to \ZZ_p$ is given by multiplication by an element of $\ZZ_p$; this is because such an additive function is necessarily continuous for the $p$-adic topology, and thus determined on the dense subgroup $\ZZ \subset \ZZ_p$.  To finish the proof, we consider two cases:

\textbf{Case 1: $p$ odd.}  Here, we may write $\alpha_p(x) = \frac{x^p +(p-1)x}{p} + ax$ for some $a \in \ZZ_p$.  Plugging in $x=1$ and using that $|BC_p|=1$ for odd $p$ (say by \Cref{thm:main}), we have by \Cref{rem:alpha_cardinality} that
\[
1+ a = \alpha(1) = |BC_p| =1,
\] 
from which we conclude $a=0$, as desired.

\textbf{Case 2: $p=2$.}  Here, we deduce immediately from the above remarks that any additive function $\ZZ_2[\varepsilon]/(2\varepsilon, \varepsilon^2) \to \ZZ_2[\varepsilon]/(2\varepsilon, \varepsilon^2)$ takes the form 
\[
r +d \varepsilon \mapsto ru + dv
\] 
for some $u,v\in \ZZ_2[\varepsilon]/(2\varepsilon, \varepsilon^2)$ (note that $v$ will be divisible by $\varepsilon$).  Thus, we may write
\[
\alpha_2(r+d\varepsilon) = \frac{r^2 +r}{2} + (rd+r+d)\varepsilon + ru + dv.
\] 
Combining our formula of $K(1)$-local cardinalities (\Cref{thm:main}) and the interaction of $\alpha_2$ with cardinalities (\Cref{rem:alpha_cardinality}), we deduce that 
\[
1 + \varepsilon + u = \alpha_2(1) = |BC_2| = 1+\varepsilon,  
\]
while
\[
1 + \varepsilon + u + v = \alpha_2(1+\varepsilon) = \alpha_2(|BC_2|) =
|B(C_2\wr C_2)| = 1+3\varepsilon = 1+\varepsilon.  
\]
This implies that $u=v =0$, as desired. 
\end{proof}

\subsubsection{$\alpha$ and symmetric bilinear forms}
The determination of $\alpha_2$ in \Cref{alpha_trick} is based on the ``coincidence'' that $|BC_2|=1+\varepsilon$. Indeed, if it happened that $|BC_2|=1$, the above method (of using the functional equation) would not have sufficed to determine $\alpha_2$. 
We now present a theoretically more involved yet more systematic approach to the computation of $\alpha_2$, based on the multiplicative aspects of our higher semiadditive Grothendieck-Witt theory. 

By \Cref{exm:qfRcmon1}, $\QF(\FF_{\ell})$ lifts to an object in $\calg(\CMon{1}{2}(\Spc))$.  We start by computing the operation $\alpha_2$ here, on $\pi_0 \QF(\FF_\ell)$, for an odd prime $\ell$.  


\begin{prop}
The total power map $\PP_2 \colon \pi_0 \QF(\FF_\ell) \to \pi_0\QF^{BC_2}(\FF_\ell)$ 
is given by 
\[
\PP_2(V,b)= (V\otimes V,b\otimes b).
\]
Here, $C_2$ acts on $V\otimes V$ by permuting the tensor factors.
\end{prop}
\begin{proof}
The space $\QF(\FF_\ell)$ is 1-truncated and hence, for $(V,b)\in \pi_0\QF(\FF_\ell)$ the map \[
\PP_2(V,b)\colon BC_2 \to \QF(\FF_\ell)
\]
is completely determined by its restriction to the $1$-skeleton $\RR\PP^1 \simeq S^1$ of $BC_2$. The restriction of the total power $\PP_2(x)$ of an element $x$ in a commutative ring $R$ to the 1-skeleton of $BC_2$ is described as follows: it takes the basepoint of $S^1$ to the square $x^2$ and the 1-cell of the circle to the symmetry that swaps the two $x$-factors using the commutativity of the ring $R$. 

In the special case $x=(V,b) \in \pi_0 \QF(\FF_\ell)$, we see that the basepoint of $S^1$ goes to $(V\otimes V,b\otimes b)$ and the 1-cell to the involution $\sigma(v\otimes u) = u\otimes v$, by the definition of the commutative ring structure of $\QF(\FF_\ell)$ (cf. \Cref{subsub:qf-fin-field}). This gives us the desired formula for $\PP_2(V,b)$. 
\end{proof}

Using the above, we may now write 
\[
\alpha_2(V,b) = \left((V\otimes V)_{C_2},\push{BC_2}(b\otimes b)\right) = \left(\Sym^2(V),\push{BC_2}(b\otimes b)\right)
\]
where, by \Cref{formula_push_quad}, we have the formula 
\begin{equation}\label{eqn:int_sym_pow_formula}
\push{BC_2}(b\otimes b)(u\otimes v,u'\otimes v') = b(u,u')b(v,v') + b(u,v')b(u',v). 
\end{equation}

Recall that a symmetric bilinear form $b$ over $\FF_\ell$ is completely determined by its rank and determinant (cf. \Cref{exm:finitefield}). Hence, to describe $\push{BC_2}(b\otimes b)$ is suffices to compute its rank and determinant. For this, we have:

\begin{prop}
Let $\ell$ be an odd prime, let $(V,b)$ be a symmetric bilinear form over $\FF_\ell$, and let $r = \dim(V)$.   Then, 
\[
    \rank(\alpha_2(V,b)) = {r+1 \choose 2}
\]
and 
\[
    \det(\alpha_2(V,b)) = 2^r\cdot\det(V,b)^{r-1} \qin \FF_\ell^\times / (\FF_\ell^\times)^2
\]
\end{prop}

\begin{proof}
The first statement amounts to the fact that the rank of $\alpha_2(V,b)$ is the dimension of $\Sym^2(V)$ which equals ${r+1 \choose 2}$. It remains to compute the determinant of $\push{BC_2}b\otimes b$. Choose a basis $e_1,...,e_r$ of $V$ which is orthogonal with respect to $b$, and set $ \lambda_i := b(e_i,e_i)$. Then, $\Sym^2(V)$ has a basis given by $\{e_i \cdot e_j\}_{i\le j}$. By formula (\ref{eqn:int_sym_pow_formula}), 
we immediately see that these basis elements $e_i\cdot e_j$ are pairwise orthogonal. 
Moreover, we have

\[
    \push{BC_2}(b\otimes b)(e_i\cdot e_j,e_i\cdot e_j) =\begin{cases} 
    \lambda_i \lambda_j & \text{ if } i \neq j \\ 
    2\lambda_i^2          & \text{ if } i=j.
    \end{cases}
\]
Consequently, we have 
\[
\det(\alpha_2(V,b)) = \prod_{i=1}^r 2\lambda_i^2\cdot \prod_{i<j}\lambda_i \lambda_j \equiv 2^r \prod_{i=1}^r \lambda_i^{r+1} = 2^r \det(V,b)^{r+1} \qin \FF_\ell^\times / (\FF_\ell^\times)^2 
\]

\end{proof}


We can now present an alternative proof for the formula of $\alpha_2$.

\begin{proof}[Proof of \Cref{alpha_trick} at $p=2$]
Choose a prime $\ell \equiv 3,5 \pmod 8$ and
consider the diagram 
\[
\xymatrix{
\pi_0 \QF(\FF_\ell)\ar^{\PP_2}[d] \ar[r] & \pi_0\GW(\FF_\ell)\ar^{\PP_2}[d]\ar^{L_{K(1)}}[r]        & \pi_0 \Sph_{K(1)}\ar^{\PP_2}[d] \\ 
\pi_0 \QF^{BC_2}(\FF_\ell)\ar^{\push{BC_2}}[d]\ar[r]  & \pi_0\GW(\FF_\ell)^{BC_2} \ar^{\push{BC_2}}[d] \ar^{L_{K(1)}}[r] & \pi_0 \Sph_{K(1)}^{BC_2}\ar^{\push{BC_2}}[d] \\  
\pi_0 \QF(\FF_\ell)\ar[r] & \pi_0\GW(\FF_\ell) \ar^{L_{K(1)}}[r]  & \pi_0 \Sph_{K(1)}
}
\]
The upper squares commute because the horizontal maps, given by group completion $\QF(\FF_\ell) \to \Omega^{\infty}\GW(\FF_{\ell})$ and $K(1)$-localization, induce maps of commutative semi-rings in $\Spc$, and hence intertwine the total power map. The lower squares commute because the maps
are maps of $2$-typical $1$-commutative pre-monoids (cf. \Cref{sk1main}). We deduce that the outer square commutes, or in other words, that we have a commutative square 
\[
\xymatrix{
\pi_0\QF(\FF_\ell) \ar^{\alpha_2}[d] \ar^\nu[r]  & \pi_0 \Sph_{K(1)} \ar^{\alpha_2}[d]\\ 
\pi_0\QF(\FF_\ell) \ar^\nu[r]  & \pi_0 \Sph_{K(1)}.\\
}
\]
Now, let $(V,b)$ be an $r$-dimensional vector space over $\FF_\ell$ with a symmetric bilinear form $b$ of determinant $2^d$. Then, since $2$ is not a quadratic residue mod $\ell$, the horizontal maps (which we have labeled $\nu$) take $(V,b)$ to $r + d\varepsilon$. Using our formula for the determinant and rank of $\alpha_2(V,b)$ we now get: 
\[
\alpha_2(r+d\varepsilon ) = \alpha_2(\nu(V,b)) = \nu(\alpha_2(V,b)) = {r+1\choose 2} + ((r+1)d + r)\varepsilon = 
\frac{r^2+r}{2} + (r + d +rd)\varepsilon,   
\]
for elements of the sub-semiring $\NN[\varepsilon]/(2\varepsilon, \varepsilon^2)\subseteq \ZZ_2[\varepsilon]/(2\varepsilon, \varepsilon^2)$ spanned by those elements of the form $r+d\varepsilon$ where $r,d\in \NN$. By \Cref{rem:functional_equation_alpha} and the lemma following this proof, we see that $\alpha_2$ is continuous with respect to the $2$-adic topology on $\ZZ_2[\varepsilon]/(2\varepsilon, \varepsilon^2)$; thus, because $\NN[\varepsilon]/(2\varepsilon, \varepsilon^2)$ is dense in $\ZZ_2[\varepsilon]/(2\varepsilon, \varepsilon^2)$, we conclude that our formula must hold everywhere in $\ZZ_2[\varepsilon]/(2\varepsilon, \varepsilon^2)$.   
\end{proof}

\begin{lem}\label{lem:alphacont}
Let $R$ be a discrete commutative $\mathbb{Z}_p$-algebra and assume that $\gamma: R\to R$ is a function satisfying the functional equation
\[
\gamma(x+y) - \gamma(x) -\gamma(y) =  \frac{(x+y)^p - x^p - y^p}{p}.
\]
Then $\gamma$ is continuous with respect to the $p$-adic topology on $R$.  
\end{lem}
\begin{proof}
Setting $y=p^jh$, we compute
\[
\gamma(x + p^jh) - \gamma(x) = \gamma(p^jh) + \frac{(x+p^jh)^p - x^p - (p^jh)^p}{p}.
\]
It suffices to show that each of the two terms on the right has $p$-adic valuation which grows arbitrarily large as $j$ grows.  This is elementary for the second term, and for the first term, this can be seen directly from the formula
\[
\gamma(tx) = t\gamma(x) + \frac{t-t^p}{p}x^p
\]
for $t\in \mathbb{Q}$ \cite[Lemma 4.1.9]{TeleAmbi}.
\end{proof}

\subsection{The operations $\theta$ and $\delta_p$}  \label{sub:thetadelta}
The operation $\alpha_p$ has two closely related cousins: the operation $\delta_p$, which is defined in any symmetric monoidal higher semiadditive category \cite{TeleAmbi}, and the operation $\theta$, which was introduced by McClure for any $K(1)$-local $\mathbb{E}_{\infty}$-ring \cite{Hinfty,HopkinsK1}.  Our aim is to clarify the relationship between these operations and compute them for the $K(1)$-local sphere. We start by reviewing their definitions.  

\begin{defn}[Definition 4.3.1 \cite{TeleAmbi}]\label{defn:delta}
Let $\cC$ be a $1$-semiadditively symmetric monoidal, additive $\infty$-category. For $R\in \calg(\cC)$, let $\delta_p\colon \pi_0(R) \to \pi_0(R)$ be defined by the equation
\[
\delta_p(x) = |BC_p|x - \alpha_p(x). 
\]
\end{defn}

To define and compute the operation $\theta$, we need to introduce some notation. 
\begin{notation}\label{ntn:theta}\hfill
\begin{itemize}
    \item Let $\Sigma_p$ be the group of permutations on $p$ letters and let $e\colon \pt \to B\Sigma_p$ and $\pi\colon B\Sigma_p \to \pt$ be the inclusion of the basepoint of $B\Sigma_p$ and the projection to the point, respectively.
    \item For an object $X$ of a stable $\infty$-category $\cC$, the above induce maps $\pi\colon X[B\Sigma_p] \to X$ and $e\colon X \to X[B\Sigma_p]$, and corresponding restrictions $\pi^*\colon X \to X^{B\Sigma_p}$ and $e^*\colon X^{B\Sigma_p} \to X$.  Here, $X[B\Sigma_p]$ denotes the tensoring of $X$ with the space $B\Sigma_p$ in $\cC$. 
    \item Since $\cC$ is stable (therefore semiadditive),  we also have \emph{transfer maps} $\tr_e \colon X[B\Sigma_p] \to X$ and $\push{e}: X \to X^{B\Sigma_p}$.
\end{itemize}
\end{notation}
In the special case where $\cC = \Sp_{K(1)}$, we have:

\begin{prop}[\cite{HopkinsK1}, Lemma 3]\label{prop:pi-tr-iso}
Let $X\in \Sp_{K(1)}$. Then the map $$(\pi,\tr_e) \colon X[B\Sigma_p] \to X \oplus X$$ is an equivalence in $\Sp_{K(1)}$. 
\end{prop}

Consequently, for $R\in \calg(\Sp_{K(1)})$, we can specify an element  $x\in \pi_0R[B\Sigma_p]$ by specifying its composites with $\pi$ and $\tr_e$, which will be elements $\pi(x), \tr_e(x)\in \pi_0(R)$. 
\begin{defn}[\cite{HopkinsK1}, \S 3]
Let $R\in \calg(\Sp_{K(1)})$. Define the element $\theta \in \pi_0R[B\Sigma_p]$ by the requirements: \footnote{We note that the first requirement differs from \cite{HopkinsK1}; the formula given  there contains a typo, as can be seen from its inconsistency with the formula relating it to the Adams operation.} 
\begin{align*}
\tr_e(\theta) &= -(p-1)! \qin \pi_0(R)  \\
\pi(\theta) &= 0 \qin \pi_0(R).
\end{align*}
\end{defn}

Our next goal is to compare $\theta$ with $\alpha$, culminating in \Cref{formula_theta}.  

\begin{rem}\label{rem:thetatranspose}
For our comparison, it will be convenient to think about the dual situation: namely, by linear duality, we have maps
\begin{align*}
\Map_R(R, R[B\Sigma_p]) &\to \Map_R(R^{B\Sigma_p}, R)\\
\Map_R(R[B\Sigma_p], R) &\to \Map_R(R,R^{B\Sigma_p})
\end{align*}
which we will denote by $f\mapsto f^t$.  Here, $\Map_R$ denotes the space of maps in $K(1)$-local $R$-modules, and these maps are equivalences because $B\Sigma_p$ is $K(1)$-locally dualizable (\cite[Theorem 8.6]{HoveyStrickland}, \cite{HopkinsK1}).  This duality acts as follows for the maps described in \Cref{ntn:theta}:
\begin{itemize}
    \item We have $e^t = e^*: R^{B\Sigma_p}\to R$ and $\pi^t = \pi^*:R \to R^{B\Sigma_p}$.  
    \item We may write $\tr_e$ as the composite
    \[
    R[B\Sigma_p] \xrightarrow{\mathrm{Nm}_{B\Sigma_p}} R^{B\Sigma_p} \xrightarrow{e^*} R.
    \]
    Here, $\mathrm{Nm}_{B\Sigma_p}$ denotes the additive norm map for the group $\Sigma_p$; note that since it comes from a symmetric pairing $R[B\Sigma_p \times B\Sigma_p] \to R$, it is linearly self-dual.  Therefore, $(\tr_e)^t$ is given by the composite:
    \[
    R \xrightarrow{e} R[B\Sigma_p] \xrightarrow{\mathrm{Nm}_{B\Sigma_p}} R^{B\Sigma_p}
    \]
    which is, by definition, $\push{e}.$
\end{itemize}
\end{rem}

We proceed by rewriting $\theta^t$ in terms of maps more closely related to $\alpha$ and $\delta_p$.  Recall that $i\colon BC_p \to B\Sigma_p$ is induced by the inclusion of a $p$-Sylow subgroup, and let $i^*\colon R^{B\Sigma_p} \to R^{BC_p}$ denote the corresponding restriction.  

\begin{prop}\label{theta_in_terms_of_alpha}
Let $R\in \calg(\Sp_{K(1)})$. 
Then we have 
\[
\theta^t = \frac{\int_{BC_p} i^* - |BC_p| e^*}{p-1} \qin \Map(R^{B\Sigma_p},R).
\]
\end{prop}

\begin{proof}
Combining \Cref{rem:thetatranspose} and \Cref{prop:pi-tr-iso}, it suffices to check this equation after precomposition with $\pi^t = \pi^*$ and $(\tr_e)^t = \push{e}$.  Thus, by definition of $\theta$, it is enough to verify the pair of identities:
\[
    \frac{\int_{BC_p}i^* \pi^* - |BC_p| e^* \pi^*}{p-1} = 0,
\]
and
\[
    \frac{\int_{BC_p}i^*\push{e} - |BC_p| e^* \push{e}}{p-1} = -(p-1)!.
\]
Now, we claim that 
\begin{enumerate}
    \item $\push{BC_p}i^* \pi^* = |BC_p|$.  This follows from the fact that $i^*\pi^*$ is restriction along the terminal map $BC_p\to \pt$ and the definition of the cardinality. 
    \item $e^* \pi^* = 1$.  This follows from the fact that $\pi \circ e = \Id_\pt$. 
    \item $\int_{BC_p}i^*\push{e} = |\Sigma_p/C_p| = (p-1)!$. This follows from the base-change formula for integrals (see \cite[Proposition 3.1.13]{TeleAmbi}) applied to the pullback square 
    \[
    \xymatrix{
    \Sigma_p / C_p \ar[d]\ar[r] & \pt \ar^e[d] \\ 
    BC_p \ar^i[r] & B\Sigma_p .
    }
    \]
    Indeed, from this pullback square we deduce that $i^*\push{e} = |\Sigma_p/C_p|\push{e'}$ where $e'$ is the inclusion of the basepoint of $BC_p$. The claim then follows because
    \[
    |\Sigma_p / C_p|\push{BC_p}\push{e'} = |\Sigma_p / C_p|
    \]
    by the Fubini Theorem for integration (see \cite[Proposition 2.1.15]{TeleAmbi}). 
    \item  $e^* \push{e} = |\Sigma_p| = p!$. This again follows from base-change of integration along the pullback square 
    \[
    \xymatrix{
    \Sigma_p \ar[d] \ar[r] & \pt \ar[d] \\ 
    \pt \ar[r]      & B\Sigma_p.
    }
    \]
\end{enumerate}

Using $(1)-(4)$ above and the fact that $p|BC_p|=p$ for all primes $p$ (which follows directly from our computation of $|BC_p|$), we see that
\[
    \frac{\int_{BC_p}i^* \pi^* - |BC_p| e^* \pi^*}{p-1} = \frac{|BC_p| - |BC_p|}{p-1} = 0  
\]
and
\[
    \frac{\int_{BC_p}i^*\push{e} - |BC_p| e^* \push{e}}{p-1} = 
    \frac{(p-1)! - |BC_p| \cdot p!}{p-1} = \frac{(p-1)! - p!}{p-1} = -(p-1)!
\]
as we wanted to show. 
\end{proof}

We can use the map $\theta^t$ to define a power operation by the composition $\theta^t \circ \PP_p$. We shall abuse notation and denote the resulting power operation by $\theta$. 


\begin{cor}\label{formula_theta}
As a power operation for $K(1)$-local $\EE_\infty$-rings, we have 
\[
\theta(x) = \frac{\alpha_p(x) - |BC_p|x^p}{p-1}.
\]
\end{cor}

\begin{proof}
Recall that $\alpha_p(x)= \push{BC_p}i^*\PP_p(x)$ (\Cref{defn:alpha}).  Thus, by \Cref{theta_in_terms_of_alpha}, 
\[
\theta(x) = \frac{\push{BC_p}i^*\PP_p(x) - |BC_p|e^*\PP_p(x)}{p-1} = \frac{\alpha_p(x) - |BC_p|x^p}{p-1}. 
\]
\end{proof}

As a result, we can now compute both the operations $\delta_p$ and $\theta$ for the $K(1)$-local sphere, as they are expressed in terms of $\alpha_p$ and $|BC_p|$.  We remark that the computations are elementary and well-known for odd primes, and may also be known to experts in general; we state it for all primes for the sake of completeness:  

\begin{thm}
For $p$ an odd prime, the power operations $\delta_p$ and $\theta$ are given on $\pi_0 \Sph_{K(1)}\simeq \ZZ_p$ by 
\[
\delta_p(x) = \theta(x) =   
\frac{x - x^p}{p}. 
\]
For $p=2$, these power operations are given on $\pi_0\Sph_{K(1)}$ by 
\[
\delta_2(r + d\varepsilon)  = 
\frac{r-r^2}{2} + rd \varepsilon
\]
and
\[
\theta(r + d\varepsilon) = \frac{r-r^2}{2} + (1+r)d\varepsilon.
\]
\end{thm}

\begin{proof}
At an odd prime, we have $\alpha_p(x) = \frac{x^p-x}{p} + x$ (\Cref{alpha_trick}) and $|BC_p| = 1$. Hence, using the formulas of \Cref{defn:delta} and \Cref{formula_theta}, we get 
\[
\delta_p(x) = |BC_p|x - \alpha_p(x)  = \frac{x - x^p}{p}
\]
and
\[
\theta(x) = \frac{\alpha_p(x) - |BC_p|x^p}{p-1} = \frac{\frac{x^p -x}{p} + x - x^p}{p-1} 
= \frac{x-x^p}{p}.
\]
At the prime $2$ we have $\alpha(r + d\varepsilon) = \frac{r^2+r}{2} + (r + d + rd)\varepsilon$ (\Cref{alpha_trick}) and $|BC_2| = 1+\varepsilon$. Hence, we get
\[
\delta_2(r + d \varepsilon) = (1+\varepsilon)(r + d\varepsilon) - \left(\frac{r^2+r}{2} + (r + d + rd)\varepsilon\right)  
= 
\frac{r - r^2}{2} +  rd \varepsilon. 
\]
Here, we used the facts that $2\varepsilon = \varepsilon^2 = 0$. 
Similarly, using the fact that $r-r^2\equiv 0 \pmod 2$, we get 
\begin{align*}
\theta(r + d \varepsilon) &=  \left(\frac{r^2+r}{2} + (r + d + rd)\varepsilon\right) - (1+\varepsilon)(r+d\varepsilon)^2 = \\
&= \frac{r-r^2}{2} + (r-r^2 + (1+r)d)\varepsilon = \frac{r-r^2}{2} +  (1+r)d\varepsilon,
\end{align*}
as we wanted to show.
\end{proof}


\subsection{The Rezk logarithm} \label{subsub:log}
For a $K(n)$-local $\EE_\infty$-ring $R$, let $R^\times$ (sometimes denoted $\gl_1(R)$) denote the connective spectrum of units in $R$. In \cite{RezkLog}, Rezk defined a ``logarithmic power operation'' 
\[
\log_{K(n)} \colon R^\times \to R,
\]
which arises as the composite of the $K(n)$-localization map $R^\times \to L_{K(n)}R^\times$ with a certain equivalence $ L_{K(n)}R^\times \simeq L_{K(n)}R$ supplied by the Bousfield-Kuhn functor (we refer the reader to \cite{RezkLog} for a more extensive discussion).  In addition to defining this logarithmic operation, Rezk computes it (on $\pi_0$) for Morava $E$-theory of any height, and for any $K(1)$-local $\EE_{\infty}$-ring spectrum $R$ satisfying the technical condition that the kernel of the unit map $\pi_0 \Sph_{K(1)} \to \pi_0 R$ contains the torsion subgroup of $\pi_0 \Sph_{K(1)}$ (which is automatically satisfied for odd $p$).  

Here, we refine this computation by computing the map
\[
\log_{K(1)} \colon \pi_0 \Sph_{K(1)}^\times \to \pi_0\Sph_{K(1)}
\]
in the case $p=2$.  For the sake of completeness, we state the result at all primes $p$:

\begin{thm}[Rezk \cite{RezkLog} for $p$ odd]\label{thm:bodylog}
Let $\log_p \colon 1+ p \ZZ_p \to \ZZ_p$ denote the $p$-adic logarithm.

\begin{itemize}
    \item For an odd prime $p$, the Rezk logarithm is given on $\pi_0\Sph_{K(1)}^\times\simeq \ZZ_p^\times$ by the formula 
\[
\log_{K(1)}(x) = \frac{1}{p}\log_p(x^{p-1}).
\]
    \item For $p=2$, the Rezk logarithm is given on $\pi_0\Sph_{K(1)}^{\times}
    \simeq \ZZ_2^\times\oplus \ZZ_2\varepsilon$ by 
    \[
    \log_{K(1)}(r + d\varepsilon) = \frac{1}{2}\log_2(r) + \frac{r-1}{2} \varepsilon.
    \]
\end{itemize}
\end{thm}

We note that when $p=2$, most of this computation is carried out by Clausen in \cite[Proposition 1.10]{DustinLog}; namely, one has an isomorphism 
\[
\pi_0 \Sph_{K(1)}^\times \simeq (\ZZ_2[\varepsilon]/(2\varepsilon, \varepsilon^2)])^{\times} \simeq \ZZ_2^{\times} \oplus \ZZ/2\ZZ \varepsilon,
\]
and Clausen computes $\log_{K(1)}$ on the $\ZZ_2^{\times}$ component.  In this section, we extend the computation to $\ZZ/2\ZZ \varepsilon$-component by showing that the logarithm vanishes on $|BC_2|=1+\varepsilon$.   

\begin{defn}
Let $R$ be an $\EE_\infty$-ring. A \tdef{strict unit} in $R$ is a map of connective spectra $\ZZ \to R^\times$. 
\end{defn}
By abuse of language, we will say a unit $x \in \pi_0 R^{\times}$ is a strict unit if there is a strict unit $\alpha\colon \ZZ \to R^\times$ such that $x = \alpha(1)$.  

\begin{lem}\label{log_kills_strict_units}
Let $n\ge 1$ and let $R$ be a $K(n)$-local $\EE_\infty$-ring. 
Then the map 
\[
\log_{K(n)} \colon \pi_0R^\times \to \pi_0R
\]
vanishes on strict units. 
\end{lem}

\begin{proof}
Let $\alpha \colon \ZZ \to R^\times$ be a strict unit. To show that $\log_{K(n)}(\alpha(1))=0$, it would suffice to show that the composite $\log_{K(n)} \circ \alpha$ vanishes in $\Map_{\Sp}(\ZZ,R)$. But $R$ is $K(n)$-local and $\ZZ$ is $K(n)$-acyclic, so there are no nonzero spectrum maps $\ZZ \to R$.
\end{proof}

In view of this lemma, to show that $\log_{K(1)}(|BC_2|)=0$, it would suffice to show that $|BC_2|$ is a strict unit.  In fact, this follows from the following general phenomenon, which we learned from Tomer Schlank:

\begin{prop}[T. Schlank] \label{card_strict}
Let $p$ be a prime and let $K(n)$ be a Morava $K$-theory of height $n$ at the prime $p$. Then
the cardinality $|B^nC_p|_{\Sph_{K(n)}}\in \pi_0\Sph_{K(n)}$ is a strict unit. 
\end{prop}

\begin{proof}
If $n=0$ then $\Sph_{K(0)}\simeq \QQ$ and every unit is a strict unit. Assume from now on that $n\ge 1$. 

Since $\Sph_{K(n)}$ is an $n$-commutative monoid in $\Sp_{K(n)}$, we have a canonical map of semirings 
\[
|-|_{\Sph_{K(n)}}\colon (\Spc_{n})^\simeq \to \Omega^\infty{\Sph_{K(n)}}
\]
taking an $n$-truncated $\pi$-finite space $A$ to the $K(n)$-local cardinality of $A$ (here, $\Spc_n$ denotes the $\infty$-category of $n$-truncated $\pi$-finite spaces). We also have a map of commutative monoids 
\[
B^n \colon \Vect_{\FF_p}^\simeq \to (\Spc_n)^\simeq
\]
taking a finite dimensional $\FF_p$-vector space $V$ to the space $B^nV$, where the monoid operation on the source is direct sum of vector spaces and on the target the Cartesian product. 
The composite 
\begin{equation}\label{eqn:cardcomp}
|B^n(-)|_{\Sph_{K(n)}}\colon \Vect_{\FF_p}^\simeq \to (\Spc_n)^\simeq \to \Omega^\infty \Sph_{K(n)}
\end{equation}
is hence again a morphism of commutative monoids, where the commutative monoid structure on $\Omega^\infty \Sph_{K(n)}$ is given by \emph{multiplication}. Since $\Sp_{K(n)}$ is an $\infty$-category of semiadditive height $n$ (in the sense of \cite{AmbiHeight}), the cardinality $|B^nV|$ is invertible for every $V\in \Vect_{\FF_p}$ (cf. \cite[Theorem 4.4.5]{AmbiHeight}) and so (\ref{eqn:cardcomp}) determines a map of spectra
\[
K(\FF_p):=(\Vect_{\FF_p}^\simeq)^{\gp} \to \Sph_{K(n)}^\times,
\]
which we denote again by $|B^n(-)|_{\Sph_{K(n)}}$. 

To show that $|B^nC_p|$ is a strict unit,
it would suffice to show that $|B^n(-)|_{\Sph_{K(n)}}$ factors through the 0-truncation map $K(\FF_p)\to \pi_0K(\FF_p)\simeq \ZZ$. For this, it is enough to show that $\Map(\tau_{\ge 1}K(\FF_p), \Sph_{K(n)}^\times)$ is contractible. Note that 
\begin{equation}\label{eqn:mapping}
\Map(\tau_{\ge 1}K(\FF_p),\Sph_{K(n)}^\times)\simeq \Map(\tau_{\ge 1}K(\FF_p),\tau_{\ge 1}\Sph_{K(n)}^\times).
\end{equation}
In the source, $p$ acts invertibly by Quillen's computation of the higher $K$-groups of $\FF_p$ \cite{Quillen}.  On the other hand, we claim the target is $p$-complete.  Recall that a spectrum is $p$-complete if and only if each homotopy group is derived $p$-complete \cite[Proposition 2.5]{Bousfield}.  Thus, since $\Sph_{K(n)}$ is $p$-complete for $n\geq 1$ and $\pi_i \Sph_{K(n)} \simeq \pi_i \Sph_{K(n)}^{\times}$ for $i\geq 1$, it follows that $\tau_{\geq 1}\Sph_{K(n)}^{\times}$ is also a $p$-complete spectrum.  Together, these facts imply that the mapping space (\ref{eqn:mapping}) is contractible and the result follows. 
\end{proof}

As a corollary, we have 
\begin{cor} \label{card_log_van}
At the prime $2$, the element $1+\varepsilon \in \pi_0\Sph_{K(1)}$ is a strict unit.
\end{cor} 
\begin{proof}
By \Cref{card_BC_2}, we have $1+\varepsilon = |BC_2|$, and so the result follows from \Cref{card_strict}. 
\end{proof}

We are now ready to complete the computation of $\log_{K(1)}$ for $\pi_0 \Sph_{K(1)}$ at $p=2$:

\begin{proof}[Proof of \Cref{thm:bodylog}]
If $p$ is odd, this is \cite[Theorem 1.9]{RezkLog}, and if $p=2$ and $d=0$ this formula is shown in \cite[Proposition 1.10]{DustinLog}. Hence, it remains to show that $\log(r+\varepsilon) = \log(r)$ for $r+\varepsilon \in \pi_0(\Sph_{K(1)})^\times$. Note that the invertibility of $r+\varepsilon$ implies that $r$ is odd. Hence, we have 
\[
(1+\varepsilon)(r+\varepsilon) = r + (r+1)\varepsilon = r.
\]
This, together with \Cref{card_log_van} implies that 
\[
\log_{K(1)}(r) = \log_{K(1)}\left((1+\varepsilon)(r+\varepsilon)\right) = \log_{K(1)}(1+\varepsilon) + \log_{K(1)}(r+\varepsilon) = \log_{K(1)}(r+\varepsilon). 
\]
\end{proof}

\bibliographystyle{alpha}
\phantomsection\addcontentsline{toc}{section}{\refname}
\bibliography{cycloref}

\newcommand{\etalchar}[1]{$^{#1}$}
\begin{thebibliography}{CDH{\etalchar{+}}20b}

\bibitem[AS69]{AtiyahSegal}
M.~F. Atiyah and G.~B. Segal.
\newblock Equivariant {$K$}-theory and completion.
\newblock {\em J. Differential Geometry}, 3:1--18, 1969.

\bibitem[AT69]{AtiyahTall}
M.~F. Atiyah and D.~O. Tall.
\newblock Group representations, {$\lambda $}-rings and the {$J$}-homomorphism.
\newblock {\em Topology}, 8:253--297, 1969.

\bibitem[Bar17]{Barwick}
Clark Barwick.
\newblock Spectral {M}ackey functors and equivariant algebraic {$K$}-theory
  ({I}).
\newblock {\em Adv. Math.}, 304:646--727, 2017.

\bibitem[BGS19]{BarwickMackey2}
Clark Barwick, Saul Glasman, and Jay Shah.
\newblock Spectral mackey functors and equivariant algebraic k-theory, ii.
\newblock {\em Tunisian Journal of Mathematics}, 2(1):97--146, 2019.

\bibitem[BMMS86]{Hinfty}
R.~R. Bruner, J.~P. May, J.~E. McClure, and M.~Steinberger.
\newblock {\em {$H_\infty $} ring spectra and their applications}, volume 1176
  of {\em Lecture Notes in Mathematics}.
\newblock Springer-Verlag, Berlin, 1986.

\bibitem[BMS21]{ShayTomer}
Shay Ben~Moshe and Tomer Schlank.
\newblock Semiadditive {A}lgebraic {K}-{T}heory and {R}edshift.
\newblock {\em In preparation}, 2021.

\bibitem[Bou79]{Bousfield}
A.~K. Bousfield.
\newblock The localization of spectra with respect to homology.
\newblock {\em Topology}, 18(4):257--281, 1979.

\bibitem[CDH{\etalchar{+}}20a]{nine-author}
Baptiste Calm{\`e}s, Emanuele Dotto, Yonatan Harpaz, Fabian Hebestreit, Markus
  Land, Kristian Moi, Denis Nardin, Thomas Nikolaus, and Wolfgang Steimle.
\newblock Hermitian k-theory for stable $\infty$-categories i: Foundations.
\newblock {\em arXiv preprint arXiv:2009.07223}, 2020.

\bibitem[CDH{\etalchar{+}}20b]{nine-author2}
Baptiste Calmes, Emanuele Dotto, Yonatan Harpaz, Fabian Hebestreit, Markus
  Land, Kristian Moi, Denis Nardin, Thomas Nikolaus, and Wolfgang Steimle.
\newblock Hermitian k-theory for stable $\infty$-categories ii: Cobordism
  categories and additivity.
\newblock {\em arXiv preprint arXiv:2009.07224}, 2020.

\bibitem[CDH{\etalchar{+}}20c]{nine-author3}
Baptiste Calmes, Emanuele Dotto, Yonatan Harpaz, Fabian Hebestreit, Markus
  Land, Kristian Moi, Denis Nardin, Thomas Nikolaus, and Wolfgang Steimle.
\newblock Hermitian k-theory for stable $\infty$-categories iii:
  Grothendieck-witt groups of rings.
\newblock {\em arXiv preprint arXiv:2009.07225}, 2(3), 2020.

\bibitem[Cla17]{DustinLog}
Dustin Clausen.
\newblock A {K}-theoretic approach to {A}rtin maps.
\newblock {\em Preprint, https://arxiv.org/abs/1703.07842}, 2017.

\bibitem[Cra11]{Cranch}
James Cranch.
\newblock Algebraic theories, span diagrams and commutative monoids in homotopy
  theory.
\newblock {\em arXiv preprint arXiv:1109.1598}, 2011.

\bibitem[CSY18]{TeleAmbi}
Shachar Carmeli, Tomer~M Schlank, and Lior Yanovski.
\newblock Ambidexterity in chromatic homotopy theory.
\newblock {\em arXiv preprint arXiv:1811.02057}, 2018.

\bibitem[CSY20]{AmbiHeight}
Shachar Carmeli, Tomer~M Schlank, and Lior Yanovski.
\newblock Ambidexterity and height.
\newblock {\em arXiv preprint arXiv:2007.13089}, 2020.

\bibitem[CSY21]{Cyclotomic}
Shachar Carmeli, Tomer~M Schlank, and Lior Yanovski.
\newblock Chromatic cyclotomic extensions.
\newblock {\em arXiv preprint arXiv:2103.02471}, 2021.

\bibitem[FHM82]{FHM}
Z.~Fiedorowicz, H.~Hauschild, and J.~P. May.
\newblock Equivariant algebraic {$K$}-theory.
\newblock In {\em Algebraic {$K$}-theory, {P}art {II} ({O}berwolfach, 1980)},
  volume 967 of {\em Lecture Notes in Math.}, pages 23--80. Springer,
  Berlin-New York, 1982.

\bibitem[Fri76]{Friedlander}
Eric~M. Friedlander.
\newblock Computations of {$K$}-theories of finite fields.
\newblock {\em Topology}, 15(1):87--109, 1976.

\bibitem[Gla16]{SaulDay}
Saul Glasman.
\newblock Day convolution for {$\infty$}-categories.
\newblock {\em Math. Res. Lett.}, 23(5):1369--1385, 2016.

\bibitem[GS96]{GS96}
J.~P.~C. Greenlees and Hal Sadofsky.
\newblock The {T}ate spectrum of {$v_n$}-periodic complex oriented theories.
\newblock {\em Math. Z.}, 222(3):391--405, 1996.

\bibitem[Har20]{HarpazSpans}
Yonatan Harpaz.
\newblock Ambidexterity and the universality of finite spans.
\newblock {\em Proceedings of the London Mathematical Society},
  121(5):1121--1170, 2020.

\bibitem[Hen17]{Henn}
Hans-Werner Henn.
\newblock A {M}ini-{C}ourse on {M}orava stabilizer groups and their cohomology.
\newblock {\em arXiv preprint arXiv:1702.05033}, 2017.

\bibitem[HL13]{HL}
Michael Hopkins and Jacob Lurie.
\newblock Ambidexterity in {$K(n)$}-local stable homotopy theory.
\newblock {\em preprint, available at https://www.math.ias.edu/~lurie/}, 2013.

\bibitem[HLAS20]{HeineLopezSpitzweck}
Hadrian Heine, Alejo Lopez-Avila, and Markus Spitzweck.
\newblock Infinity categories with duality and hermitian multiplicative
  infinite loop space machines.
\newblock {\em arXiv preprint arXiv:1610.10043}, 2020.

\bibitem[HMS20]{ShiftedCo}
Rune Haugseng, Valerio Melani, and Pavel Safronov.
\newblock Shifted coisotropic correspondences.
\newblock {\em Journal of the Institute of Mathematics of Jussieu}, pages
  1--65, 2020.

\bibitem[Hop14]{HopkinsK1}
Michael~J. Hopkins.
\newblock {$K(1)$}-local {$E_\infty$}-ring spectra.
\newblock In {\em Topological modular forms}, volume 201 of {\em Math. Surveys
  Monogr.}, pages 287--302. Amer. Math. Soc., Providence, RI, 2014.

\bibitem[HS96]{HS96}
Mark Hovey and Hal Sadofsky.
\newblock Tate cohomology lowers chromatic {B}ousfield classes.
\newblock {\em Proc. Amer. Math. Soc.}, 124(11):3579--3585, 1996.

\bibitem[HS99]{HoveyStrickland}
Mark Hovey and Neil~P Strickland.
\newblock {\em Morava $ K $-theories and localisation}, volume 666.
\newblock American Mathematical Soc., 1999.

\bibitem[Kuh04]{kuhn2004tate}
Nicholas~J Kuhn.
\newblock Tate cohomology and periodic localization of polynomial functors.
\newblock {\em Inventiones mathematicae}, 157(2):345--370, 2004.

\bibitem[Lam05]{lam2005introduction}
T.~Y. Lam.
\newblock {\em Introduction to quadratic forms over fields}, volume~67 of {\em
  Graduate Studies in Mathematics}.
\newblock American Mathematical Society, Providence, RI, 2005.

\bibitem[Lur09]{LurGood}
Jacob Lurie.
\newblock $(\infty ,2)$-{C}ategories and the {G}oodwillie {C}alculus {I}.
\newblock {\em \newline {A}vailable at http://www.math.harvard.edu/~lurie/},
  2009.

\bibitem[Lur15]{LurInt}
Jacob Lurie.
\newblock Representation theory in intermediate characteristic.
\newblock {\em Lecture notes by Tony Feng, available at
  http://www.mit.edu/people/fengt/Lurie\_plen2.pdf}, 2015.

\bibitem[Lur16]{HA}
Jacob Lurie.
\newblock Higher {A}lgebra.
\newblock {\em \newline {A}vailable at https://www.math.ias.edu/~lurie/}, 2016.

\bibitem[Lur17]{HTT}
Jacob Lurie.
\newblock Higher {T}opos {T}heory.
\newblock {\em \newline {A}vailable at https://www.math.ias.edu/~lurie/}, 2017.

\bibitem[Mah81]{Mahowald}
Mark Mahowald.
\newblock {$b{\rm o}$}-resolutions.
\newblock {\em Pacific J. Math.}, 92(2):365--383, 1981.

\bibitem[Mil81]{Miller}
Haynes~R. Miller.
\newblock On relations between {A}dams spectral sequences, with an application
  to the stable homotopy of a {M}oore space.
\newblock {\em J. Pure Appl. Algebra}, 20(3):287--312, 1981.

\bibitem[Qui72]{Quillen}
Daniel Quillen.
\newblock On the cohomology and {$K$}-theory of the general linear groups over
  a finite field.
\newblock {\em Ann. of Math. (2)}, 96:552--586, 1972.

\bibitem[Rez01]{Rezk}
Charles Rezk.
\newblock A model for the homotopy theory of homotopy theory.
\newblock {\em Trans. Amer. Math. Soc.}, 353(3):973--1007, 2001.

\bibitem[Rez06]{RezkLog}
Charles Rezk.
\newblock The units of a ring spectrum and a logarithmic cohomology operation.
\newblock {\em J. Amer. Math. Soc.}, 19(4):969--1014, 2006.

\bibitem[Sch18]{schwede2018global}
Stefan Schwede.
\newblock {\em Global homotopy theory}, volume~34.
\newblock Cambridge University Press, 2018.

\end{thebibliography}

\end{document}